\documentclass[11pt,reqno]{amsart}

\usepackage{amsmath,amssymb,amsthm,mathtools}
\usepackage{enumitem}
\usepackage{microtype}
\usepackage{xcolor}
\usepackage[colorlinks=true,linkcolor=blue!55!black,citecolor=blue!55!black,urlcolor=blue!55!black]{hyperref}
\usepackage[nameinlink,capitalise,noabbrev]{cleveref}

\numberwithin{equation}{section}

\newtheorem{theorem}{Theorem}[section]
\newtheorem{proposition}[theorem]{Proposition}
\newtheorem{lemma}[theorem]{Lemma}
\newtheorem{corollary}[theorem]{Corollary}

\theoremstyle{definition}
\newtheorem{definition}[theorem]{Definition}

\theoremstyle{remark}
\newtheorem{remark}[theorem]{Remark}

\newcommand{\HH}{\mathbb H}
\newcommand{\CC}{\mathbb C}
\newcommand{\RR}{\mathbb R}
\newcommand{\SSph}{\mathbb S}
\newcommand{\cA}{\mathcal A}

\newcommand{\cC}{\mathcal C}

\newcommand{\cL}{\mathcal L}
\newcommand{\cR}{\mathcal R}
\newcommand{\cT}{\mathcal T}
\newcommand{\cH}{\mathcal H}
\newcommand{\cM}{\mathcal M}
\newcommand{\fA}{\mathfrak A}
\newcommand{\GammaH}{\Gamma_{\!H}}
\newcommand{\dV}{\,dV_S}
\newcommand{\dx}{\,d\xi}
\newcommand{\dy}{\,d\eta}
\newcommand{\loc}{\mathrm{loc}}
\newcommand{\supp}{\operatorname{supp}}

\newcommand{\Span}{\operatorname{span}}
\newcommand{\Id}{\mathrm{Id}}

\newcommand{\ii}{\mathrm i}

\newcommand{\weakto}{\rightharpoonup}

\DeclareMathOperator{\Rea}{Re}
\DeclareMathOperator{\Ima}{Im}

\setlist[itemize]{leftmargin=2em,itemsep=0.25em,topsep=0.35em}
\setlist[enumerate]{leftmargin=2.2em,itemsep=0.25em,topsep=0.35em}

\allowdisplaybreaks

\title[Classification for the CR Yamabe equation on $\HH^n$]{Classification of positive entire solutions of the CR Yamabe equation on the Heisenberg group}
\author{Jungang Li}
\address{University of Science and Technology of China}
\email{jungangli@ustc.edu.cn}
\thanks{The author was supported by the National Natural Science Foundation of China (NSFC), grant no.~12571127.}
\date{August 7, 2026}

\hypersetup{
  pdftitle={Classification of positive entire solutions of the CR Yamabe equation on the Heisenberg group},
  pdfauthor={Jungang Li},
  pdfsubject={Unconditional classification for the critical CR Yamabe equation},
  pdfkeywords={CR Yamabe equation, Heisenberg group, Liouville theorem, Jerison--Lee identity, Green representation, concentration compactness, defect quantization, Morrey estimates}
}

\begin{document}

\begin{abstract}
We prove that, for every $n\ge2$, every positive entire solution of the
critical CR Yamabe equation
\[
 4\Delta_bu=n^2u^{(Q+2)/(Q-2)},\qquad Q=2n+2,
\]
on the Heisenberg group $\HH^n$ is a Jerison--Lee bubble.  No
integrability, decay, boundedness, or symmetry is assumed.  Together
with the theorem of Catino, Li, Monticelli, and Roncoroni in $\HH^1$,
this classifies the positive entire solutions in every dimension.

Both Euclidean routes to such a statement lose their starting
configuration here.  Hyperplane reflections are not CR automorphisms,
and a CR inversion preserves its Kor\'anyi sphere only setwise, so the
difference between a solution and its Kelvin transform need not vanish
on the sphere one inverts in.  Nor is there a substitute a priori bound
to fall back on: the only scale-invariant estimate available for every
positive solution is a critical Morrey bound, which concentration
saturates and which yields neither decay nor a small-mass regularity
principle.

We proceed instead from two exact consequences of the Green
representation, which every positive solution is shown to satisfy.
Reciprocity with a bubble $U$ converts the distance from $U$ into a
nonnegative convex deficit, so that no information about the sign of a
linearized quadratic form is required; differentiating the same
reciprocity along the conformal orbit of $U$ gives a nonlinear
barycentre identity, which forbids that deficit from concentrating at a
single point of the CR sphere.  Together these isolate the bubble
manifold in a class carrying no energy bound.  Quantizing the
Jerison--Lee tensor defect against this isolation, and alternating the
resulting budget with the Morrey bound, drives the energy-growth
exponent into the range where the defect must vanish.  Both identities
use only conformal covariance and an exact positive Green
representation.
\end{abstract}

\maketitle

\section{Introduction}

Let $\HH^n=\CC^n\times\RR$ be the Heisenberg group and let
$Q=2n+2$ be its homogeneous dimension.  In the standard
pseudohermitian normalization, the critical CR Yamabe equation is
\begin{equation}\label{eq:intro-equation}
 4\Delta_bu=n^2u^p,
 \qquad
 p=\frac{Q+2}{Q-2}=1+\frac2n,
 \qquad u>0\quad\hbox{in }\HH^n.
\end{equation}
The equation is invariant under left translations, Heisenberg
dilations, and CR inversion.  Its finite-energy solutions are the
extremals of the sharp Folland--Stein Sobolev inequality.  Jerison and
Lee proved that every positive solution of \eqref{eq:intro-equation}
in $L^{p+1}(\HH^n)$ belongs to an explicit conformal family
\cite{JerisonLee1988}, in the course of their work on the CR Yamabe
problem \cite{JerisonLee1987}; the sharp inequality and its conformal
formulation were subsequently revisited by Frank and Lieb
\cite{FrankLieb2012}.

The Euclidean analogue was settled without an assumption at infinity
by Caffarelli, Gidas, and Spruck and, by a different integral
moving-plane argument, by Chen and Li, building on the classical
maximum-principle symmetry framework
\cite{GidasNiNirenberg1979,CaffarelliGidasSpruck1989,ChenLi1991}.
For $n\ge2$, the corresponding unconditional statement on $\HH^n$
has remained outside the finite-energy theory.  We prove it; see
\cref{thm:main}.

\subsection{Geometric obstruction and the role of finite energy}

The Euclidean comparison methods do not transfer verbatim.  General
hyperplane reflections are not CR automorphisms of $\HH^n$; the
adapted reflections that remain available underlie important symmetry
results, but do not reproduce the full Euclidean moving-plane sweep
without additional structure
\cite{BirindelliPrajapat1999,GarofaloVassilev2001}.  The method of
moving spheres of Li and Zhu \cite{LiZhu1995}, developed further in
\cite{LiZhang2003,Li2004}, replaces reflections by Kelvin transforms
and is in many Euclidean situations the more flexible of the two.  On
$\HH^n$, however, a CR Kelvin inversion preserves its Kor\'anyi sphere
setwise but need not fix that sphere pointwise, so the usual zero
boundary datum for the difference of a function and its Kelvin
transform is absent.  This does not rule out a moving-spheres
argument, but it means that its initial comparison must use an
additional idea.  In the integral formulation of these methods
\cite{Li2004,ChenLiOu2006}, symmetry results on $\HH^n$ were obtained
by Prajapat and Varghese \cite{PrajapatVarghese2025}; the integral
representation will play a different role below, as the source of the
identities of \cref{sec:riesz-morrey}.

Nor is finite energy merely a convenient normalization.  On the
punctured group, Afeltra constructed singular and dilation-periodic
solutions displaying genuine scale recurrence
\cite{Afeltra2020,Afeltra2026}.  These examples are not entire, but
they show that the scale-rich behavior excluded by the variational
hypothesis is real.  At the same time, the natural estimate enjoyed by
every entire solution is critical: we prove
\[
 \int_{B_R(a)}u^p\le CR^n
\]
at every centre and scale, yet a concentrating bubble has unbounded
height while respecting exactly this bound.  Thus critical Morrey
control alone gives neither a pointwise estimate nor a small-mass
regularity principle.

In $\HH^1$, Catino, Li, Monticelli, and Roncoroni obtained the
unconditional classification
\cite{CatinoLiMonticelliRoncoroni2025}.  The same paper classifies the
solutions for every $n\ge2$ under the pointwise decay hypothesis
\[
 u(\xi)\le C\bigl(1+\rho(\xi)\bigr)^{-n}.
\]
Independently, Flynn and V\'etois obtained the classification for every
$n\ge2$ under the weaker pointwise hypothesis
$u(\xi)\le C\rho(\xi)^{-(n-2)}$, and also under the integral condition
\[
 \int_{B_R}u^q\le CR^2
 \qquad\hbox{for some }
 q\in\Bigl(\tfrac{2n+1}{n},\tfrac{2n+2}{n}\Bigr];
\]
in particular every bounded solution is classified when $n=2$
\cite{FlynnVetois2023}.  For the endpoint case of their pointwise
hypothesis they use an argument from
\cite{CatinoLiMonticelliRoncoroni2025}; see
\cite[Remark~1.2]{FlynnVetois2023} for a detailed comparison of the two
sets of hypotheses.  \Cref{thm:main} removes all of them.

These works and the present paper use the Jerison--Lee divergence
identity as an algebraic engine, and the argument below is indebted to
both; what differs is how the initial global control is obtained.  A
different route to the same unconditional statement, by a method of
moving spheres, has been proposed by Liu \cite{Liu2025movingspheres};
the argument given here is independent of it.  After the present paper
was circulated, X.-N. Ma informed the author that he, J.-H. Li, and
Y.-L. Liu had independently obtained the same unconditional
classification, by a different argument based on a dilation-charge
interaction exclusion and a decomposition into natural-scale packets
\cite{LiLiuMa2026}.

The use of a nonnegative divergence defect has parallels in the
critical $p$-Laplace equation
\cite{CatinoMonticelliRoncoroni2023,Ou2025}; related integral estimates
for semilinear equations on $\HH^n$ were developed by Ma and Ou
\cite{MaOu2023}, and the structural origin of the Jerison--Lee family of
differential identities is analysed by Ma, Ou, and Wu \cite{MaOuWu2024}.
The rigidity endpoint is the CR counterpart of the classical conformal
rigidity principle of Obata \cite{Obata1971}.

Fix the basic bubble
\begin{equation}\label{eq:basic-bubble-intro}
 U_0(z,t)
 =2^n\bigl((1+|z|^2)^2+t^2\bigr)^{-n/2}.
\end{equation}
For $a\in\HH^n$ and $\lambda>0$, set
\begin{equation}\label{eq:bubble-family-intro}
 U_{a,\lambda}(\xi)
 =\lambda^{-n}U_0\!\left(
       \delta_{\lambda^{-1}}(a^{-1}\xi)\right).
\end{equation}
This is the full Jerison--Lee family in the normalization
\eqref{eq:intro-equation}.

\subsection{Main results}

\begin{theorem}[Unconditional classification]\label{thm:main}
Let $n\ge2$ and let $u\in C^2(\HH^n)$ be a positive solution of
\eqref{eq:intro-equation}.  Then there exist $a\in\HH^n$ and
$\lambda>0$ such that
\[
 u=U_{a,\lambda}.
\]
In particular, $u\in L^{p+1}(\HH^n)$, although no integrability,
decay, boundedness, or symmetry is assumed a priori.
\end{theorem}

Thus finite energy is a conclusion rather than a hypothesis.  Combining
\cref{thm:main} with the result in $\HH^1$ gives the all-dimensional
statement.

\begin{corollary}\label{cor:all-dimensional}
For every $n\ge1$, the positive entire solutions of
\eqref{eq:intro-equation} are exactly the Jerison--Lee bubbles
\eqref{eq:bubble-family-intro}.
\end{corollary}

The proof yields two intermediate results with independent content.
The first isolates the bubble manifold in the class to which, by
\cref{sec:riesz-morrey}, every positive entire solution belongs.

\begin{theorem}[Sequential bubble isolation]\label{thm:intro-isolation}
Let $n\ge2$, and let $(W_j)$ be positive entire solutions of
\eqref{eq:intro-equation} with a uniform exact Green representation
and uniform all-centre critical Morrey bounds.  If
\[
 W_j\longrightarrow U
 \quad\hbox{in }C^\infty_{H,\loc}(\HH^n)
\]
for a Jerison--Lee bubble $U$, then $W_j$ is itself a Jerison--Lee
bubble for all sufficiently large $j$.
\end{theorem}

The hypotheses of \cref{thm:intro-isolation} involve no energy bound
of any kind.  We expect the theorem to be relevant to a priori
estimates near isolated blow-up points for Yamabe-type equations on CR
manifolds, in the sense of \cite[Remark~1.3]{FlynnVetois2023} and by
analogy with the Riemannian compactness theory
\cite{KhuriMarquesSchoen2009}.

The second turns a geometric defect budget into an improved energy
budget.  Let $\fA_f\ge0$ denote the Jerison--Lee density defined in
\cref{sec:jl-defect}, and put
\[
 J_a(R)=\int_{B_R(a)}\fA_f.
\]

\begin{theorem}[Defect budget and energy improvement]
\label{thm:intro-improvement}
Let $W$ be a bounded positive exact-Riesz/Morrey solution of
\eqref{eq:intro-equation} which is not a bubble.  Suppose that, for
some $d\in[0,n-2]$,
\[
 \sup_{a\in\HH^n}J_a(R)\le CR^d
 \qquad(R\ge1).
\]
Then
\begin{equation}\label{eq:intro-energy-improvement}
 \sup_{a\in\HH^n}\int_{B_R(a)}W^{p+1}
 \le C'R^{\alpha(d)},
 \qquad
 \alpha(d)=\frac{n(d+2)}{n+2}.
\end{equation}
Consequently, an energy-growth exponent $D>2$ improves to
$nD/(n+2)$; when $D\le2$, the Jerison--Lee defect vanishes.
\end{theorem}

\subsection{Two exact identities}

The first global input is itself unconditional.  With
$L=4\Delta_b=-\Delta_H$ and $\GammaH$ its positive fundamental
solution, every positive entire solution satisfies
\begin{equation}\label{eq:intro-riesz}
 u(\xi)=n^2\int_{\HH^n}
 \GammaH(\eta^{-1}\xi)u(\eta)^p\dy.
\end{equation}
There is no additive harmonic term: positivity makes such a remainder
constant, and a positive constant would make the Green potential
infinite.  Positivity of the same kernel then gives the all-centre
critical Morrey estimate
\begin{equation}\label{eq:intro-morrey}
 \int_{B_R(a)}u^p\le CR^n.
\end{equation}
Both statements are preserved by the affine conformal action.

Now let $U$ be a bubble, let $W$ be a solution in this class, and write
$F=W/U$.  Symmetry of the Green kernel and Tonelli's theorem give
\[
 \int_{\HH^n}U^pW=\int_{\HH^n}UW^p.
\]
For
\[
 \cR_p(s)=(1+s)^p-1-ps,
\]
this reciprocity is equivalent to the exact formula
\begin{equation}\label{eq:intro-deficit}
 \mathfrak D_U(W)
 :=\int U^{p+1}-\int U^pW
 =\frac1{p-1}\int U^{p+1}
        \cR_p\!\left(\frac WU-1\right)\ge0.
\end{equation}
All integrals in \eqref{eq:intro-deficit} are over $\HH^n$.  The
nonnegativity comes from strict convexity rather than the sign of a
linearized quadratic form.  Only in the special case $p=2$ does this
reduce to the square deficit $\int U(W-U)^2$.  The statement and its
short proof are recorded in \cref{cor:convex-deficit}.  Deficits of
this shape also occur in the quantitative stability theory for the
Folland--Stein inequality \cite{ChenFanLiao2025}; here
$\mathfrak D_U$ is used only as an exact algebraic quantity, and no
stability inequality enters the argument.

The second identity appears after Cayley compactification.  On the CR
sphere, write again $F=1+f$.  Differentiating Green reciprocity along
the $Q$-dimensional conformal orbit of the constant solution gives
\begin{equation}\label{eq:intro-barycenter}
 \int_{\SSph^{2n+1}}X_\alpha\cR_p(f)\dV=0,
 \qquad \alpha=1,\ldots,Q.
\end{equation}
This is a nonlinear barycentre law for a nonnegative density.  It is
also exactly compatible with the bubble family: reciprocity between
any two bubbles gives \eqref{eq:intro-barycenter} without an
asymptotic expansion.

Suppose now that non-bubbles converge locally to one bubble.  After
modulation transverse to the conformal orbit, the normalized convex
deficit can only escape through the point at infinity of the Cayley
chart.  It would therefore converge to a Dirac mass.  Equation
\eqref{eq:intro-barycenter} forbids this, because a point of the unit
sphere cannot have all $Q$ coordinate functions equal to zero.  This
is the sequential isolation mechanism in \cref{thm:intro-isolation}.

\subsection{Quantization and finite descent}

For a bounded solution, Harnack inequalities and relative interior
estimates control the Jerison--Lee defect pointwise.  The $m=0$
divergence identity, integrated against a cutoff, gives the annular
estimate
\begin{equation}\label{eq:intro-annular}
 J_a(R)^2
 \le CR^{-2}\left(
  \int_{B_{2R}(a)}W^{p+1}
  +R^{-4}\int_{B_{2R}(a)}W^{2-2/n}
 \right)\bigl(J_a(2R)-J_a(R)\bigr).
\end{equation}
A discrete Riccati argument turns an energy budget into a defect
budget.  Conversely, sequential isolation implies a uniform defect
quantum on every doubling-good natural-scale ball.  A height-layer
covering then combines that quantum with \eqref{eq:intro-morrey} and
proves \eqref{eq:intro-energy-improvement}.

The two budgets alternate.  If
$\int_{B_R(a)}W^{p+1}\le CR^D$, then
\eqref{eq:intro-annular} gives a defect budget of dimension $D-2$,
and \cref{thm:intro-improvement} returns the new energy dimension
\[
 D\longmapsto\frac{n}{n+2}D.
\]
Starting from the elementary bounded-solution estimate $D_0=n$,
finitely many steps reach $D\le2$.  The annular inequality then forces
$\fA_f\equiv0$, and the pointwise rigidity system identifies $W$ as
a bubble.  An enhanced doubling argument reduces the unbounded case to
the bounded classification and sequential isolation.

\subsection{Scope and algebraic inputs}

At the level of identities, \eqref{eq:intro-deficit} and
\eqref{eq:intro-barycenter} use only conformal covariance and an exact
positive Green representation.  They therefore have direct Euclidean
analogues.  On $\HH^n$, the critical exponent is $p=1+2/n$, so
$p\le2$ holds automatically for every $n\ge2$ and imposes no further
restriction on the present theorem.  For the Euclidean critical
exponent, the same condition holds only when $N\ge6$.  In this respect
the CR problem is more favourable; the Euclidean dimensions
$N=3,4,5$ would require a different uniform-integrability argument.

The two algebraic inputs are located explicitly.  First,
$\fA_f\equiv0$ implies that the solution is a Jerison--Lee bubble;
\cref{prop:zero-defect-rigidity} proves this on all of $\HH^n$,
without decay or integrability.  The only external local ingredient in
that rigidity step is the standard characterization of
CR-pluriharmonic functions
\cite{Bedford1980,Lee1988pseudoEinstein}.  This is also where the
assumption $n\ge2$ is structural: when $n=1$, the trace-free tensor
$E_{\alpha\bar\beta}$ is identically zero.  Second,
\cite[Proposition~4.1, formula~(4.2)]{JerisonLee1988} supplies the
pointwise $m=0$ divergence identity.  \Cref{app:jl-cutoff}
translates that identity into our conventions and proves the complete
cutoff argument leading to \eqref{eq:intro-annular}.  No growth
hypothesis or Liouville theorem from \cite{FlynnVetois2023} enters
that argument.

\subsection{Organization}

The logical order is strictly
\cref{sec:riesz-morrey} $\to$ \cref{sec:sphere-isolation} $\to$
\cref{sec:defect-quantization} $\to$ \cref{sec:dimension-descent}.
In particular, the proof of \cref{thm:intro-isolation} uses only the
exact Green representation and the Morrey bound of
\cref{sec:riesz-morrey}; it uses nothing from \cref{sec:jl-defect}
onwards.  In \cref{prop:defect-quantum} the limiting profile is
identified as a bubble by the pointwise rigidity of
\cref{prop:zero-defect-rigidity}, not by the bounded classification of
\cref{thm:bounded-profile}.  There is therefore no circularity in the
alternation of the two budgets.

\Cref{sec:preliminaries} fixes the geometry and compactness tools.
\Cref{sec:riesz-morrey} proves the exact Riesz representation, the
Morrey bound, and Green reciprocity.  \Cref{sec:sphere-isolation}
proves sequential bubble isolation.  \Cref{sec:jl-defect} introduces
the tensor defect and the annular estimate.
\Cref{sec:defect-quantization} proves the uniform quantum and the
height-layer energy improvement, while \cref{sec:dimension-descent}
completes the iteration and proves \cref{thm:main}.  The appendices
record the conformal normalizations and zero-defect rigidity, the
standard subelliptic inputs, and the complete $m=0$ cutoff argument
from the Jerison--Lee identity.

\subsection*{Acknowledgements}

The author thanks Congwen Liu for helpful correspondence on the
geometry of the CR inversion on $\HH^n$; Xi-Nan Ma for telling him
about the independent work with Jiahuan Li and Yilu Liu, and for
sharing their manuscript; Yanyan Li for correspondence that improved
the account of the earlier higher-dimensional results; and Joshua Flynn
for his comments on an earlier version.

\section{Geometry, normalization, and compactness tools}\label{sec:preliminaries}

\subsection{The Heisenberg group}

We use the group law
\begin{equation}\label{eq:group-law}
 (z,t)(w,s)
 =\bigl(z+w,t+s+2\Ima(z\cdot\overline w)\bigr)
 \qquad(z,w\in\CC^n,\ t,s\in\RR).
\end{equation}
The anisotropic dilations are
\[
 \delta_r(z,t)=(rz,r^2t),
 \qquad r>0,
\]
and Haar measure scales by $|\delta_rE|=r^Q|E|$, where $Q=2n+2$.  Writing $z_j=x_j+\ii y_j$, the standard real horizontal fields are
\[
 X_j=\partial_{x_j}+2y_j\partial_t,
 \qquad
 Y_j=\partial_{y_j}-2x_j\partial_t.
\]
We set
\begin{equation}\label{eq:L-def}
 \Delta_H=\sum_{j=1}^n(X_j^2+Y_j^2),
 \qquad
 L=-\Delta_H=4\Delta_b.
\end{equation}
Thus \eqref{eq:intro-equation} becomes
\begin{equation}\label{eq:main-equation}
 Lu=n^2u^p,
 \qquad p=1+\frac2n.
\end{equation}
Any positive distributional solution is smooth by the usual subelliptic bootstrap, so no distinction between classical and smooth positive solutions will be needed below.

The Kor\'anyi gauge and its left-invariant distance are
\begin{equation}\label{eq:koranyi}
 \rho(z,t)=\bigl(|z|^4+t^2\bigr)^{1/4},
 \qquad
 d(\xi,\eta)=\rho(\eta^{-1}\xi).
\end{equation}
With the normalization in \eqref{eq:group-law}, this is a genuine
left-invariant metric, not merely a quasi-distance; see
\cite{Cygan1981}.  In particular, the triangle inequality may be used
without changing constants.
We write $B_R(a)=\{\xi:d(a,\xi)<R\}$.  Replacing $d$ by any fixed homogeneous distance changes only dimensional constants.  In particular,
\begin{equation}\label{eq:volume-growth}
 |B_R(a)|=\omega_QR^Q.
\end{equation}

The positive fundamental solution of $L$ is
\begin{equation}\label{eq:fundamental-solution}
 \GammaH(\xi)=c_Q\rho(\xi)^{2-Q},
 \qquad
 L\GammaH=\delta_e,
\end{equation}
for a dimensional constant $c_Q>0$ \cite{Folland1973}.  We shall repeatedly use the two-sided estimate
\begin{equation}\label{eq:green-kernel-estimate}
 \GammaH(\eta^{-1}\xi)\asymp d(\xi,\eta)^{2-Q}.
\end{equation}
The value of $c_Q$ here is rescaled to the operator $L$ in \eqref{eq:L-def}; this avoids importing the factor $1/4$ present in some conventions for the Heisenberg sub-Laplacian.

\subsection{Conformal affine action and bubbles}

For $a\in\HH^n$ and $\mu>0$, define
\begin{equation}\label{eq:rescaling}
 (\cT_{a,\mu}u)(\zeta)
 =\mu^nu(a\delta_\mu\zeta).
\end{equation}
The homogeneity of $L$ and the identity $n(p-1)=2$ imply that \eqref{eq:main-equation} is invariant under $\cT_{a,\mu}$.  Its inverse is
\[
 (\cT_{a,\mu}^{-1}v)(\xi)
 =\mu^{-n}v\!\left(\delta_{\mu^{-1}}(a^{-1}\xi)\right).
\]
The function $U_0$ in \eqref{eq:basic-bubble-intro} solves \eqref{eq:main-equation}; its affine orbit is exactly \eqref{eq:bubble-family-intro}.  We denote this $Q$-dimensional manifold by
\[
 \cM_{\rm bub}
 =\{U_{a,\lambda}:a\in\HH^n,\ \lambda>0\}.
\]
Membership in $\cM_{\rm bub}$ is invariant under every $\cT_{a,\mu}$.

\begin{definition}\label{def:riesz-morrey-class}
A family $(u_j)$ of positive solutions of \eqref{eq:main-equation} is called a \emph{uniform exact-Riesz/Morrey family} if
\begin{equation}\label{eq:uniform-riesz}
 u_j(\xi)
 =n^2\int_{\HH^n}\GammaH(\eta^{-1}\xi)u_j(\eta)^p\dy
\end{equation}
for every $j$ and $\xi$, and if one constant $C_M$ satisfies
\begin{equation}\label{eq:uniform-morrey}
 \int_{B_R(a)}u_j^p\le C_MR^n
 \qquad(a\in\HH^n,\ R>0,\ j\ge1).
\end{equation}
The same terminology is used for a single solution.
\end{definition}

The change of variables in \eqref{eq:rescaling} shows that \eqref{eq:uniform-riesz}--\eqref{eq:uniform-morrey} are invariant, with the same Morrey constant, under $\cT_{a,\mu}$.

\subsection{An enhanced doubling lemma}

We record the metric selection lemma in the form needed later.  It is the standard doubling argument of Pol\'a\v cik, Quittner, and Souplet \cite{PolacikQuittnerSouplet2007}, with the displacement estimate retained from its proof.

\begin{lemma}[Doubling with displacement]\label{lem:doubling}
Let $(X,d)$ be a complete metric space and let $M:X\to(0,\infty)$ be locally bounded.  Given $x\in X$ and $k>0$, there exists $y\in X$ such that
\begin{equation}\label{eq:doubling-displacement}
 M(y)\ge M(x),
 \qquad
 d(x,y)\le\frac{2k}{M(x)},
\end{equation}
and
\begin{equation}\label{eq:doubling-good}
 M(z)\le2M(y)
 \qquad\text{whenever }d(z,y)\le\frac{k}{M(y)}.
\end{equation}
\end{lemma}

\begin{proof}
If \eqref{eq:doubling-good} holds at $x$, take $y=x$.  Otherwise choose $x_1$ with
\[
 d(x,x_1)\le kM(x)^{-1},
 \qquad M(x_1)>2M(x).
\]
If the conclusion again fails, iterate.  After $j$ steps,
\[
 M(x_j)>2^jM(x),
 \qquad
 d(x_j,x_{j+1})\le k2^{-j}M(x)^{-1}.
\]
If the procedure did not stop, $(x_j)$ would be Cauchy and contained in the compact set formed by the sequence and its limit, whereas $M(x_j)\to\infty$, contradicting local boundedness.  At the stopping point $y=x_N$, \eqref{eq:doubling-good} holds, and
\[
 d(x,y)\le\sum_{j=0}^{N-1}d(x_j,x_{j+1})
 \le\frac{k}{M(x)}\sum_{j=0}^\infty2^{-j}
 =\frac{2k}{M(x)}.
\]
\end{proof}

For a positive solution $u$, we apply \cref{lem:doubling} to
\[
 M=u^{1/n}.
\]
If $\mu=M(y)^{-1}$, the natural rescaling
\begin{equation}\label{eq:natural-rescaling}
 V(\zeta)=\mu^nu(y\delta_\mu\zeta)
\end{equation}
satisfies
\begin{equation}\label{eq:natural-good}
 V(e)=1,
 \qquad
 V\le2^n\quad\text{on }B_k(e).
\end{equation}

\subsection{Local compactness on expanding good balls}

We use the following standard consequence of Harnack inequalities and interior estimates for sums of squares; the analytic references are collected in \cref{app:analytic-inputs}.

\begin{lemma}[Expanding-window compactness]\label{lem:expanding-compactness}
Let $u_j$ solve \eqref{eq:main-equation} and suppose
\[
 u_j(e)=1,
 \qquad
 0<u_j\le C_0\quad\text{on }B_{R_j}(e),
 \qquad R_j\to\infty.
\]
Then, after passing to a subsequence,
\[
 u_j\longrightarrow u_\infty
 \quad\text{in }C^\infty_{H,\loc}(\HH^n),
\]
where $u_\infty$ is a positive entire solution, $u_\infty(e)=1$, and $u_\infty\le C_0$.
\end{lemma}

\begin{proof}
For each fixed $R$, Harnack gives a positive lower bound for $u_j$ on $B_R$ once $j$ is large.  Interior subelliptic estimates and a bootstrap in \eqref{eq:main-equation} give uniform $C_H^{m,\alpha}(B_R)$ bounds for every $m$.  A diagonal Arzel\`a--Ascoli argument proves the assertion.
\end{proof}

\section{Exact Riesz representation and the critical Morrey bound}\label{sec:riesz-morrey}

This section supplies the global information available for every positive entire solution.  No decay or integrability assumption is used.

\subsection{Elimination of the harmonic remainder}

\begin{theorem}[Exact Green--Riesz representation]\label{thm:exact-riesz}
Let $u>0$ be a smooth entire solution of \eqref{eq:main-equation}.  Then
\begin{equation}\label{eq:exact-riesz}
 u(\xi)=n^2\int_{\HH^n}\GammaH(\eta^{-1}\xi)u(\eta)^p\dy<\infty
 \qquad(\xi\in\HH^n).
\end{equation}
\end{theorem}

\begin{proof}
Choose an increasing exhaustion $\Omega_R\nearrow\HH^n$ by bounded regular domains, and let $G_R$ be the positive Dirichlet Green kernel for the Friedrichs realization of $L$ on $\Omega_R$.  Put
\[
 v_R(\xi)=n^2\int_{\Omega_R}G_R(\xi,\eta)u(\eta)^p\dy.
\]
Weak comparison gives $0\le v_R\le u$ in $\Omega_R$.  Indeed, $v_R$ has zero trace and $u$ is strictly positive up to the compact boundary, so $(v_R-u)_+\in S^1_{H,0}(\Omega_R)$.  Testing $L(v_R-u)=0$ with this function gives
\[
 \int_{\Omega_R}|\nabla_H(v_R-u)_+|^2=0.
\]
Domain monotonicity gives $v_R\le v_S$ when $R<S$.

The killed heat kernels increase to the full heat kernel, and transience in homogeneous dimension $Q>2$ yields
\[
 G_R(\xi,\eta)\uparrow\GammaH(\eta^{-1}\xi)
 \qquad(\xi\ne\eta).
\]
Consequently, monotone convergence gives a finite potential
\[
 v(\xi)=n^2\int_{\HH^n}\GammaH(\eta^{-1}\xi)u(\eta)^p\dy
 \le u(\xi).
\]
Since $v\in L^1_{\loc}$ and $v\le u$, Fubini's theorem is legitimate against compactly supported test functions, and $Lv=n^2u^p$ in distributions.  Hence
\[
 h=u-v\ge0,
 \qquad Lh=0\quad\text{in }\HH^n.
\]
Hypoellipticity makes $h$ smooth.  The Liouville theorem for nonnegative entire $L$-harmonic functions implies that $h\equiv c$ for some $c\ge0$; see \cite{BonfiglioliLanconelli2001}.  If $c>0$, then $u^p\ge c^p$ and therefore
\[
 v(\xi)
 \ge n^2c^p\int_{\HH^n}\GammaH(\eta^{-1}\xi)\dy
 \ge c'\int_1^\infty r^{2-Q}r^{Q-1}\,dr
 =\infty,
\]
contrary to $v\le u<\infty$.  Thus $c=0$ and \eqref{eq:exact-riesz} follows.
\end{proof}

\begin{remark}\label{rem:exhaustion}
The exhaustion argument is phrased through the Friedrichs Green kernels so that no boundary regularity at characteristic points of a Kor\'anyi sphere is needed.  Equivalently, one may use killed heat kernels throughout.
\end{remark}

\subsection{All-centre source control}

\begin{proposition}[Critical Morrey estimate]\label{prop:critical-morrey}
There is a dimensional constant $C_M$ such that every positive entire solution of \eqref{eq:main-equation} satisfies
\begin{equation}\label{eq:critical-morrey}
 \int_{B_R(a)}u^p\le C_MR^n
 \qquad(a\in\HH^n,\ R>0).
\end{equation}
\end{proposition}

\begin{proof}
Fix $a$ and $R$, and write
\[
 M=\int_{B_R(a)}u^p.
\]
This number is finite because $u$ is smooth on the compact closure of
$B_R(a)$; it also follows directly by restricting the finite potential
in \eqref{eq:exact-riesz} to the ball.
If $\xi,\eta\in B_R(a)$, then $d(\xi,\eta)\le2R$.  The positivity of the exact representation and \eqref{eq:green-kernel-estimate} give
\[
 u(\xi)\ge cR^{2-Q}M
 \qquad(\xi\in B_R(a)).
\]
It follows from \eqref{eq:volume-growth} that
\[
 M
 \ge |B_R|\bigl(cR^{2-Q}M\bigr)^p
 \ge c_1R^{Q+p(2-Q)}M^p
 =c_1R^{-2}M^p.
\]
Since $p-1=2/n$, this is equivalent to $M\le C_MR^n$.
\end{proof}

The Morrey estimate and the homogeneity of $\GammaH$ give a uniform weighted tail.

\begin{corollary}[Dyadic Green tail]\label{cor:green-tail}
For every $a\in\HH^n$ and $R\ge1$,
\begin{equation}\label{eq:green-tail}
 \int_{\HH^n\setminus B_R(a)}
 d(a,\eta)^{2-Q}u(\eta)^p\dy
 \le CR^{-n}.
\end{equation}
The constant is dimensional and is preserved under the rescalings \eqref{eq:rescaling}.
\end{corollary}

\begin{proof}
On $B_{2^{k+1}R}(a)\setminus B_{2^kR}(a)$, \cref{prop:critical-morrey} gives
\[
 \int d(a,\eta)^{2-Q}u^p\dy
 \le C(2^kR)^{2-Q}(2^{k+1}R)^n
 \le C(2^kR)^{-n}.
\]
Sum over $k\ge0$.
\end{proof}

\begin{remark}\label{rem:morrey-not-pointwise}
The exponent in \eqref{eq:critical-morrey} is critical.  A bubble of width $\varepsilon$ has height comparable to $\varepsilon^{-n}$ while its $u^p$-mass on a fixed ball is of order $\varepsilon^n$.  Thus the Morrey estimate does not imply a pointwise upper bound or a small-mass regularity statement.
\end{remark}

\subsection{Affine limits retain the global equation}

The tail estimate is exactly what prevents a harmonic defect from appearing when an affine sequence is passed to a local limit.

\begin{proposition}[Exact-Riesz compactness]\label{prop:riesz-compactness}
Let $(u_j)$ be a uniform exact-Riesz/Morrey family, and assume
\[
 u_j\longrightarrow u_\infty
 \quad\text{in }C^\infty_{H,\loc}(\HH^n),
 \qquad u_\infty>0.
\]
Then $u_\infty$ satisfies the same exact representation and the same all-centre Morrey bound.
\end{proposition}

\begin{proof}
The Morrey estimate passes to the limit by Fatou.  Fix $\xi\in\HH^n$ and $R>2d(e,\xi)$.  On $B_R(e)$, local convergence and the local integrability of the Green singularity allow passage to the limit in
\[
 \int_{B_R(e)}\GammaH(\eta^{-1}\xi)u_j(\eta)^p\dy.
\]
For $\eta\notin B_R(e)$, the distances $d(\xi,\eta)$ and $d(e,\eta)$ are comparable.  The uniform version of \eqref{eq:green-tail} bounds the complementary integral by $CR^{-n}$, independently of $j$.  First let $j\to\infty$ and then $R\to\infty$ in \eqref{eq:uniform-riesz}.  This gives the exact representation for $u_\infty$.
\end{proof}

\subsection{Green reciprocity with a bubble}

\begin{lemma}[Exact reciprocity]\label{lem:reciprocity}
Let $W$ be a positive exact-Riesz/Morrey solution and let $U\in\cM_{\rm bub}$.  Then
\begin{equation}\label{eq:reciprocity}
 \int_{\HH^n}U^pW\dx
 =\int_{\HH^n}UW^p\dx
 <\infty.
\end{equation}
More generally, the same identity holds between $W$ and every bubble obtained from $U$ by a sufficiently small conformal deformation on the CR sphere.
\end{lemma}

\begin{proof}
Fix $\xi_0\in\HH^n$.  Since $U$ is bounded near $\xi_0$, decays like $\rho^{2-Q}$ at infinity, and $\GammaH(\xi_0^{-1}\cdot)$ has the same decay while blowing up at $\xi_0$, one has
\[
 U(\eta)\le C\GammaH(\xi_0^{-1}\eta)
 \qquad(\eta\in\HH^n).
\]
The exact representation of $W$ at $\xi_0$ gives
\[
 \int UW^p\le C\int\GammaH(\xi_0^{-1}\eta)W(\eta)^p\dy<\infty.
\]
Apply the exact representation first to $W$ and then to $U$.  Positivity, Tonelli's theorem, and the symmetry
$\GammaH(\eta^{-1}\xi)=\GammaH(\xi^{-1}\eta)$ yield
\begin{align*}
 \int U^pW
 &=n^2\iint U(\xi)^p\GammaH(\eta^{-1}\xi)W(\eta)^p\,d\eta\,d\xi\\
 &=\int UW^p.
\end{align*}
The same proof applies to a nearby bubble.  On a compact parameter interval, the bubble, its parameter derivative, its $p$th power, and the derivative of that power satisfy the same uniform Green-kernel domination.  This is the $C^1$ domination used when reciprocity is differentiated in \cref{lem:nonlinear-barycenter}.
\end{proof}

\begin{corollary}[Convex deficit]\label{cor:convex-deficit}
Let $W$ be a positive exact-Riesz/Morrey solution and let
$U\in\cM_{\rm bub}$.  Set
$\cR_p(s)=(1+s)^p-1-ps$ for $s>-1$.  Then
\begin{align}
 \mathfrak D_U(W)
 &:=\int_{\HH^n}U^{p+1}\dx-\int_{\HH^n}U^pW\dx\notag\\
 &=\frac1{p-1}\int_{\HH^n}U^{p+1}
       \cR_p\!\left(\frac WU-1\right)\dx\ge0.
 \label{eq:convex-deficit}
\end{align}
Equality holds if and only if $W=U$.
\end{corollary}

\begin{proof}
Expanding the last integrand gives
\[
 U^{p+1}\cR_p\!\left(\frac WU-1\right)
 =UW^p+(p-1)U^{p+1}-pU^pW.
\]
After integration, \cref{lem:reciprocity} identifies the first term
with $\int U^pW\dx$ and proves \eqref{eq:convex-deficit}.  Since
$p>1$, strict convexity gives $\cR_p(s)\ge0$, with equality only at
$s=0$; hence equality in \eqref{eq:convex-deficit} is equivalent to
$W=U$.
\end{proof}

\section{Sequential isolation of the bubble manifold}\label{sec:sphere-isolation}

Throughout this section $n\ge2$, and hence $1<p\le2$.  We prove \cref{thm:intro-isolation}.  By applying one fixed affine transformation, it is enough to treat a sequence converging locally to the basic bubble $U_0$.

\subsection{Cayley compactification and the missing point}

Let
\[
 N=(0,\ldots,0,-1)\in\SSph^{2n+1}\subset\CC^{n+1}.
\]
We call $N$ the \emph{point at infinity} for the chosen Cayley
transform, avoiding a north/south convention.
We use the Cayley transform
\begin{equation}\label{eq:cayley-transform}
 \cC(z,t)
 =\left(
 \frac{2z}{1+|z|^2+\ii t},
 \frac{1-|z|^2-\ii t}{1+|z|^2+\ii t}
 \right),
\end{equation}
which maps $\HH^n$ onto $\SSph^{2n+1}\setminus\{N\}$.  Normalize the spherical volume form by the exact identity
\begin{equation}\label{eq:cayley-volume}
 \cC^*(dV_S)=U_0^{p+1}\dx
\end{equation}
under \eqref{eq:cayley-transform}.  The conformal sub-Laplacian $\cL_S$ is normalized by
\begin{equation}\label{eq:conformal-covariance}
 L\bigl(U_0(F\circ\cC)\bigr)
 =U_0^p\bigl((\cL_SF)\circ\cC\bigr).
\end{equation}
In particular,
\begin{equation}\label{eq:kappa-sphere}
 \cL_S1=\kappa_S,
 \qquad \kappa_S=n^2.
\end{equation}

Given an entire solution $W$, define its relative conformal factor by
\begin{equation}\label{eq:relative-factor}
 F(\cC\xi)=\frac{W(\xi)}{U_0(\xi)}.
\end{equation}
Then
\begin{equation}\label{eq:sphere-equation-punctured}
 \cL_SF=\kappa_SF^p
 \qquad\text{on }\SSph^{2n+1}\setminus\{N\}.
\end{equation}

\begin{lemma}[Spherical Green equation]\label{lem:global-sphere-green}
If $W$ is an exact-Riesz/Morrey solution, then $F,F^p\in L^1(\SSph^{2n+1})$ and
\begin{equation}\label{eq:global-sphere-green}
 F(\zeta)
 =\kappa_S\int_{\SSph^{2n+1}}G_S(\zeta,\eta)F(\eta)^p\dV(\eta)
\end{equation}
for almost every $\zeta$.  Equivalently, \eqref{eq:sphere-equation-punctured} holds on the whole sphere in distributions, with no multiple of $\delta_N$ and no spherical harmonic remainder.
\end{lemma}

\begin{proof}
The Green kernels are related by
\begin{equation}\label{eq:green-dictionary}
 G_S(\cC\xi,\cC\eta)
 =\frac{\GammaH(\eta^{-1}\xi)}{U_0(\xi)U_0(\eta)}.
\end{equation}
The coefficient is exactly one: it is fixed by the covariance relation \eqref{eq:conformal-covariance}, the measure normalization \eqref{eq:cayley-volume}, and $L\GammaH=\delta_e$.  Substituting \eqref{eq:relative-factor}, \eqref{eq:cayley-volume}, and \eqref{eq:green-dictionary} into the exact Riesz representation gives \eqref{eq:global-sphere-green} away from $N$.

It remains to justify the global integrability.  By \cref{lem:reciprocity},
\begin{equation}\label{eq:sphere-integrability}
 \int_{\SSph^{2n+1}}F^p\dV
 =\int_{\HH^n}U_0W^p\dx
 =\int_{\HH^n}U_0^pW\dx
 =\int_{\SSph^{2n+1}}F\dV<\infty.
\end{equation}
Thus the right-hand side of \eqref{eq:global-sphere-green} is the Green potential of an $L^1$ density on the compact sphere.  A point has zero volume, so this representation itself contains no atom at $N$.  Applying $\cL_S$ in distributions proves the last assertion.
\end{proof}

\subsection{The spherical kernel and a conformal slice}

Let $X_1,\ldots,X_Q$ be the real coordinate functions of
$\SSph^{2n+1}\subset\RR^Q$.

\begin{lemma}[First spherical eigenspace]\label{lem:sphere-kernel}
Let
\[
 \cA_S=\cL_S-p\kappa_S.
\]
Then
\begin{equation}\label{eq:sphere-kernel}
 \ker\cA_S
 =\Span_{\RR}\{X_1,\ldots,X_Q\}.
\end{equation}
Moreover, these functions are precisely the tangent directions at $1$ to the noncompact conformal orbit of the constant solution.
\end{lemma}

\begin{proof}
The bispherical decomposition is
\[
 L^2(\SSph^{2n+1})
 =\widehat\bigoplus_{r,s\ge0}\cH_{r,s},
\]
and, in the normalization \eqref{eq:conformal-covariance},
\begin{equation}\label{eq:bispherical-spectrum}
 \left.\cL_S\right|_{\cH_{r,s}}
 =4\left(r+\frac n2\right)\left(s+\frac n2\right).
\end{equation}
This follows from the spherical spectral computation in \cite{FrankLieb2012}; see also \cref{app:normalizations}.  Since $p\kappa_S=n(n+2)$, equality in \eqref{eq:bispherical-spectrum} occurs only for $(r,s)=(1,0)$ or $(0,1)$.  Their real span is generated by the $Q$ coordinate functions.  Differentiating the conformal orbit of $1$ gives the same space.
\end{proof}

We need a quantitative slice that only differentiates smooth kernels, not a possibly singular conformal factor.

\begin{lemma}[Bubble-gauge modulation]\label{lem:modulation}
Suppose $F_j>0$ satisfies the global Green equation \eqref{eq:global-sphere-green} and $F_j\to1$ in $L^1(\SSph^{2n+1})$.  Then there are CR conformal transformations $\varphi_j\to\Id$ in $C^\infty$ such that the conformally transformed factors, still denoted by $F_j$, satisfy
\begin{equation}\label{eq:modulation-orthogonality}
 \int_{\SSph^{2n+1}}X_\alpha(F_j-1)\dV=0,
 \qquad \alpha=1,\ldots,Q.
\end{equation}
The transformed factors continue to satisfy \eqref{eq:global-sphere-green}.  If the only possible nonsmooth point before modulation is $N$, after modulation it is $N_j=\varphi_j^{-1}(N)$ and $N_j\to N$.
\end{lemma}

\begin{proof}
We exhibit the required family explicitly.  Put $m=n+1$, so that
$\SSph^{2n+1}=\SSph^{2m-1}\subset\CC^m$.  For $v\in\CC^m$ close
to zero, set
\begin{equation}\label{eq:explicit-boost}
 r_v=(1+|v|^2)^{1/2},\qquad
 A_v=I+\frac{vv^*}{1+r_v},\qquad
 \varphi_v(\zeta)=\frac{A_v\zeta+v}{r_v+v^*\zeta}.
\end{equation}
The block matrix
\[
 M_v=\begin{pmatrix}A_v&v\\v^*&r_v\end{pmatrix}
\]
satisfies
\[
 M_v^*\begin{pmatrix}I&0\\0&-1\end{pmatrix}M_v
 =\begin{pmatrix}I&0\\0&-1\end{pmatrix},
 \qquad \det M_v=1.
\]
Thus $M_v\in SU(m,1)$, so $\varphi_v$ is a CR automorphism and
$\varphi_0=\Id$.  Its
Jacobian and conformal factor in our normalization are
\begin{equation}\label{eq:boost-factor}
 J_v(\zeta)=|r_v+v^*\zeta|^{-2m},
 \qquad
 B_v=J_v^{1/(p+1)}
     =|r_v+v^*\zeta|^{-n}.
\end{equation}
Define the conformal action by
\[
 \cT_vF=B_v(F\circ\varphi_v).
\]
Conformal covariance reads
\[
 \cL_S(\cT_vF)
 =B_v^p\bigl((\cL_SF)\circ\varphi_v\bigr),
\]
and hence preserves both the equation and its Green formulation.

For $q\in\CC^m$, direct differentiation of
\eqref{eq:explicit-boost} and \eqref{eq:boost-factor} gives
\begin{equation}\label{eq:boost-linearization}
 \left.\frac{d}{dt}\right|_{t=0}\varphi_{tq}(\zeta)
 =q-\zeta(q^*\zeta),
 \qquad
 \left.\frac{d}{dt}\right|_{t=0}B_{tq}(\zeta)
 =-n\Rea(q^*\zeta).
\end{equation}
Use the real parameter basis
$e_1,\ii e_1,\ldots,e_m,\ii e_m$ of $\CC^m\simeq\RR^Q$,
ordered compatibly with the coordinate functions $X_1,\ldots,X_Q$.
Then the tangent vectors to $\cT_v1=B_v$ are exactly
$-nX_1,\ldots,-nX_Q$.

Define
\[
 \Phi_\alpha(v,F)
 =\int X_\alpha(\cT_vF-1)\dV.
\]
After changing variables, each $\Phi_\alpha$ has the form
\[
 \Phi_\alpha(v,F)=\int K_{\alpha,v}F\dV-\int X_\alpha\dV,
\]
where $K_{\alpha,v}$ is smooth in both variables.  Thus the map is
$C^1$ in $v$ for $F\in L^1$.  At $(0,1)$ its real $v$-Jacobian is
\[
 \partial_{v_\beta}\Phi_\alpha(0,1)
 =-n\int X_\alpha X_\beta\dV,
\]
which is invertible.  The quantitative implicit function theorem
therefore supplies $v_j\to0$ with $\Phi(v_j,F_j)=0$.  Taking
$\varphi_j=\varphi_{v_j}$ proves the orthogonality.  The assertion
about $N_j$ follows from $\varphi_j\to\Id$.
\end{proof}

\subsection{The convex deficit and one-point escape}

We now consider a uniform exact-Riesz/Morrey sequence $W_j$ such that
\begin{equation}\label{eq:local-bubble-convergence}
 W_j\longrightarrow U_0
 \quad\text{in }C^\infty_{H,\loc}(\HH^n).
\end{equation}
Let $F_j=1+f_j$ be its spherical relative factors.

\begin{lemma}[Global $L^1$ convergence]\label{lem:global-L1}
Under \eqref{eq:local-bubble-convergence},
\begin{equation}\label{eq:F-L1}
 F_j\longrightarrow1
 \quad\text{in }L^1(\SSph^{2n+1}).
\end{equation}
If
\begin{equation}\label{eq:orlicz-remainder}
 \cR_p(f)=(1+f)^p-1-pf,
 \qquad -1<f<\infty,
\end{equation}
and
\begin{equation}\label{eq:epsilon-def}
 \varepsilon_j=\int_{\SSph^{2n+1}}\cR_p(f_j)\dV,
\end{equation}
then $\varepsilon_j\to0$.
\end{lemma}

\begin{proof}
By \eqref{eq:cayley-volume},
\[
 \int F_j\dV=\int_{\HH^n}U_0^pW_j\dx.
\]
On an annulus $A_R=B_{2R}\setminus B_R$, H\"older's inequality and the uniform Morrey bound give
\begin{equation}\label{eq:W-L1-annulus}
 \int_{A_R}W_j
 \le\left(\int_{A_R}W_j^p\right)^{1/p}|A_R|^{1-1/p}
 \le CR^{n+2}.
\end{equation}
Since $U_0^p=O(\rho^{-(Q+2)})$, dyadic summation yields
\begin{equation}\label{eq:weighted-tail-isolation}
 \sup_j\int_{\rho>R}U_0^pW_j\dx\le CR^{-(n+2)}.
\end{equation}
Local convergence and \eqref{eq:weighted-tail-isolation} prove \eqref{eq:F-L1}.

Integrating the spherical equation and using \eqref{eq:kappa-sphere} gives
\begin{equation}\label{eq:sphere-mass-equality}
 \int F_j^p\dV=\int F_j\dV.
\end{equation}
Consequently,
\begin{align}
 \varepsilon_j
 &=\int(F_j^p-1-p(F_j-1))\dV\notag\\
 &=(p-1)\left(|\SSph^{2n+1}|-\int F_j\dV\right)\longrightarrow0.
 \label{eq:epsilon-mass}
\end{align}
By \eqref{eq:cayley-volume}, \eqref{eq:relative-factor}, and
\cref{cor:convex-deficit}, this is exactly
\begin{equation}\label{eq:epsilon-deficit-dictionary}
 \varepsilon_j=(p-1)\mathfrak D_{U_0}(W_j).
\end{equation}
\end{proof}

Apply \cref{lem:modulation}; henceforth \eqref{eq:modulation-orthogonality} is assumed.  We redefine $f_j=F_j-1$ and $\varepsilon_j$ by \eqref{eq:epsilon-def} after this transformation.  Since the modulation tends to the identity, $F_j\to1$ locally on $\SSph^{2n+1}\setminus\{N\}$ and in $L^1$, while the possible bad point $N_j$ still converges to $N$.

The elementary convexity estimate used below is
\begin{equation}\label{eq:orlicz-two-scale}
 c_p\left(f^2\mathbf1_{\{|f|\le1\}}+f^p\mathbf1_{\{f>1\}}\right)
 \le\cR_p(f)
 \le C_p\left(f^2\mathbf1_{\{|f|\le1\}}+f^p\mathbf1_{\{f>1\}}\right).
\end{equation}

\begin{lemma}[Normalized defect compactness]\label{lem:normalized-defect}
Assume $\varepsilon_j>0$ and set
\[
 g_j=\frac{f_j}{\sqrt{\varepsilon_j}}.
\]
Then $(g_j)$ is uniformly integrable in $L^1(\SSph^{2n+1})$, every weak $L^1$ limit is zero, and
\begin{equation}\label{eq:local-strong-L2}
 g_j\longrightarrow0
 \quad\text{strongly in }L^2(K)
\end{equation}
for every compact $K\Subset\SSph^{2n+1}\setminus\{N\}$.
\end{lemma}

\begin{proof}
For every measurable $E$, \eqref{eq:orlicz-two-scale} and H\"older's inequality give
\begin{align}
 \int_E|g_j|\dV
 &\le C|E|^{1/2}
 +C\varepsilon_j^{1/p-1/2}|E|^{1-1/p}\notag\\
 &\le C\bigl(|E|^{1/2}+|E|^{1-1/p}\bigr),
 \label{eq:uniform-integrability}
\end{align}
where the last step uses $p\le2$ and, eventually, $\varepsilon_j\le1$.  Thus Dunford--Pettis gives weak sequential compactness in $L^1$.

Subtracting the constant equation from \eqref{eq:sphere-equation-punctured} gives globally in distributions
\begin{equation}\label{eq:linearized-defect-equation}
 \cA_Sf_j=\kappa_S\cR_p(f_j).
\end{equation}
After division by $\sqrt{\varepsilon_j}$, the right-hand side has $L^1$ norm $\kappa_S\sqrt{\varepsilon_j}\to0$.  Hence every weak limit of $g_j$ lies in $\ker\cA_S$.  The orthogonality \eqref{eq:modulation-orthogonality} and \cref{lem:sphere-kernel} force that limit to be zero.

Fix $K\Subset K'\Subset\SSph^{2n+1}\setminus\{N\}$.  On $K'$, $f_j\to0$ uniformly.  Define
\[
 a_j=
 \begin{cases}
 \cR_p(f_j)/f_j,&f_j\ne0,\\
 0,&f_j=0.
 \end{cases}
\]
Then $a_j\to0$ uniformly on $K'$ and
\[
 \cA_Sg_j=\kappa_Sa_jg_j.
\]
The small-$f$ part of \eqref{eq:orlicz-two-scale} gives a uniform $L^2(K')$ bound for $g_j$.  Interior subelliptic estimates give a uniform $S_H^{2,2}(K)$ bound, and the local Rellich theorem makes $(g_j)$ precompact in $L^2(K)$.  Any $L^2$ cluster point must agree with the already identified weak $L^1$ limit, hence is zero.  This proves \eqref{eq:local-strong-L2}.
\end{proof}

Define probability measures
\begin{equation}\label{eq:defect-probability}
 \nu_j=\varepsilon_j^{-1}\cR_p(f_j)\,dV_S.
\end{equation}

\begin{corollary}\label{cor:one-point-concentration}
Under the hypotheses of \cref{lem:normalized-defect}, the defect concentrates at one point:
\begin{equation}\label{eq:nu-to-delta}
 \nu_j\weakto\delta_N.
\end{equation}
\end{corollary}

\begin{proof}
For $K\Subset\SSph^{2n+1}\setminus\{N\}$, one has $|f_j|\le1$ for large $j$.  Thus
\[
 \nu_j(K)
 \le C\varepsilon_j^{-1}\int_Kf_j^2\dV
 =C\int_Kg_j^2\dV\longrightarrow0.
\]
The sphere is compact and $\nu_j$ has total mass one, so every weak limit is the unit mass at $N$.
\end{proof}

\subsection{The nonlinear barycentre identity}

Let $q_\alpha$ be the element of the real parameter basis used after
\eqref{eq:boost-linearization} that corresponds to $X_\alpha$, and set
$B_{\alpha,t}=B_{tq_\alpha}$.  The explicit boost family then gives
\begin{equation}\label{eq:bubble-branch-tangent}
 B_{\alpha,0}=1,
 \qquad
 \left.\frac{d}{dt}\right|_{t=0}B_{\alpha,t}=-nX_\alpha.
\end{equation}

\begin{lemma}[Nonlinear barycentre]\label{lem:nonlinear-barycenter}
Every global Green solution $F=1+f$ satisfies
\begin{equation}\label{eq:nonlinear-barycenter}
 \int_{\SSph^{2n+1}}X_\alpha\cR_p(f)\dV=0,
 \qquad \alpha=1,\ldots,Q.
\end{equation}
\end{lemma}

\begin{proof}
Green-kernel symmetry and Tonelli's theorem give exact reciprocity between $F$ and $B_{\alpha,t}$:
\begin{equation}\label{eq:sphere-reciprocity}
 \int B_{\alpha,t}^pF\dV
 =\int B_{\alpha,t}F^p\dV.
\end{equation}
For small $t$, the bubble branch and its derivative are uniformly bounded, while $F,F^p\in L^1$.  We may therefore differentiate \eqref{eq:sphere-reciprocity} at $t=0$.  Using \eqref{eq:bubble-branch-tangent} and dividing by $-n$ gives
\[
 \int X_\alpha(pF-F^p)\dV=0.
\]
Since $F^p-pF=1-p+\cR_p(f)$ and $\int X_\alpha\dV=0$, this is exactly \eqref{eq:nonlinear-barycenter}.
\end{proof}

\begin{proof}[Proof of \cref{thm:intro-isolation}]
Assume, after taking a subsequence, that no $W_j$ is a bubble.  If $\varepsilon_j=0$, strict convexity in \eqref{eq:orlicz-remainder} gives $f_j=0$ almost everywhere after modulation, hence $W_j$ is a bubble.  Therefore $\varepsilon_j>0$.

Divide \eqref{eq:nonlinear-barycenter} by $\varepsilon_j$.  In terms of \eqref{eq:defect-probability},
\begin{equation}\label{eq:zero-barycenter-probability}
 \int X_\alpha\,d\nu_j=0
 \qquad(\alpha=1,\ldots,Q).
\end{equation}
On the other hand, \cref{cor:one-point-concentration} makes the left-hand side converge to $X_\alpha(N)$.  The $Q$ coordinates of a point on the unit sphere cannot all vanish.  This contradicts \eqref{eq:zero-barycenter-probability} and proves the theorem.
\end{proof}

\section{The Jerison--Lee defect and annular growth}\label{sec:jl-defect}

In this section $W$ is a bounded positive solution of
\eqref{eq:main-equation}.  We introduce the nonnegative geometric
defect and derive a uniform annular growth budget from the $m=0$
Jerison--Lee divergence identity.

\subsection{The defect density}

Write
\begin{equation}\label{eq:f-def}
 f=\frac1n\log W-\log2.
\end{equation}
With the complex horizontal fields
\[
 Z_\alpha=\partial_{z_\alpha}+\ii\overline z_\alpha\partial_t,
 \qquad
 Z_{\bar\alpha}=\overline{Z_\alpha},
\]
we use subscripts for covariant derivatives in the standard flat pseudohermitian structure and sum over repeated Greek indices.  Following Jerison--Lee and Flynn--V\'etois, set
\begin{align}
 D_{\alpha\beta}
 &=f_{\alpha\beta}-2f_\alpha f_\beta,
 &
 E_{\alpha\bar\beta}
 &=f_{\alpha\bar\beta}
   -\frac1n f_{\gamma\bar\gamma}\delta_{\alpha\bar\beta},
 \label{eq:DE-tensors}\\
 D_\alpha
 &=D_{\alpha\beta}f_{\bar\beta},
 &
 E_\alpha
 &=E_{\alpha\bar\beta}f_\beta,
 \label{eq:DE-one-forms}\\
 g&=|\partial f|^2+e^{2f}-\ii f_0,
 &
 G_\alpha&=\ii f_{0\alpha}+gf_\alpha.
 \label{eq:gG-def}
\end{align}
The equation for $W$ is equivalent to
\begin{equation}\label{eq:f-equation}
 \Delta_bf=n|\partial f|^2+ne^{2f}.
\end{equation}

Define the $m=0$ defect density
\begin{equation}\label{eq:defect-density}
 \fA_f
 =e^{2(n-1)f}
 \left[
 e^{2f}\bigl(|D_{\alpha\beta}|^2+|E_{\alpha\bar\beta}|^2\bigr)
 +|D_\alpha|^2+|E_\alpha|^2+|G_\alpha|^2
 \right].
\end{equation}
It is nonnegative and vanishes exactly when all tensors displayed in \eqref{eq:defect-density} vanish.  We use the standard pointwise rigidity implication recorded precisely in \cref{prop:zero-defect-rigidity}:
\begin{equation}\label{eq:zero-defect-bubble}
 \fA_f\equiv0
 \quad\Longrightarrow\quad
 W\in\cM_{\rm bub}.
\end{equation}

\subsection{Uniform relative derivatives}

The highest derivative in \eqref{eq:defect-density} is $f_{0\alpha}$, of homogeneous order three because $\partial_t$ has weight two.

\begin{lemma}[Relative interior estimates]\label{lem:relative-derivatives}
For every differential monomial $D_H^I$ generated by horizontal fields and their commutators, of homogeneous order $|I|_H\le3$, there is a constant
$C_I=C_I(n,\|W\|_\infty)$ such that
\begin{equation}\label{eq:relative-u-derivatives}
 |D_H^IW(a)|\le C_IW(a),
 \qquad a\in\HH^n.
\end{equation}
Consequently,
\begin{equation}\label{eq:log-derivatives}
 |D_H^If(a)|\le C_I
 \qquad(1\le|I|_H\le3),
\end{equation}
uniformly in $a$.
\end{lemma}

\begin{proof}
Fix $a$ and normalize $v=W/W(a)$ on $B_1(a)$.  Rewrite its equation as the linear equation
\[
 Lv=n^2W(a)^{p-1}v^{p-1}v
   =n^2W^{p-1}v.
\]
The potential $n^2W^{p-1}$ is bounded by $n^2\|W\|_\infty^{p-1}$.  Since $v(a)=1$, the unit-ball Harnack inequality gives two-sided bounds for $v$ on $B_{3/4}(a)$ with constants independent of $a$.  Interior estimates for sums of squares, followed by a bootstrap in the smooth nonlinearity $v^p$, give
\[
 \|v\|_{C_H^{3,\alpha}(B_{1/2}(a))}\le C.
\]
Evaluating at $a$ proves \eqref{eq:relative-u-derivatives}.  The identities obtained by differentiating $\log W$ and the positive lower bound for $v$ give \eqref{eq:log-derivatives}.  Derivatives such as $f_{0\alpha}$ are included because $\partial_t$ is a step-two commutator of horizontal fields, so $\partial_tZ_\alpha$ has homogeneous horizontal order three.  The required Harnack and interior estimates are recalled in \cref{app:analytic-inputs}.
\end{proof}

\begin{corollary}[Pointwise defect control]\label{cor:defect-pointwise}
Set
\begin{equation}\label{eq:q-minus}
 q_-=2-\frac2n.
\end{equation}
Then
\begin{equation}\label{eq:defect-pointwise}
 \fA_f\le C W^{q_-}
\end{equation}
with $C=C(n,\|W\|_\infty)$.
\end{corollary}

\begin{proof}
The derivative factors in \eqref{eq:defect-density} are bounded by
\cref{lem:relative-derivatives}.  The same is true of $e^{2f}$.  Finally,
\[
 e^{2(n-1)f}=C_nW^{2(n-1)/n}=C_nW^{q_-}.
\]
\end{proof}

\subsection{The annular inequality}

For $a\in\HH^n$, define
\begin{equation}\label{eq:J-def}
 J_a(R)=\int_{B_R(a)}\fA_f\dx.
\end{equation}

\begin{proposition}[Jerison--Lee annular estimate]\label{prop:annular-defect}
For every $R>0$,
\begin{equation}\label{eq:annular-defect}
 \begin{aligned}
 J_a(R)^2
 &\le CR^{-2}
 \left[
 \int_{B_{2R}(a)}W^{p+1}\dx
 +R^{-4}\int_{B_{2R}(a)}W^{q_-}\dx
 \right]\\
 &\qquad\times\bigl(J_a(2R)-J_a(R)\bigr).
 \end{aligned}
\end{equation}
The constant is independent of $a$ and $R$.
\end{proposition}

\begin{proof}
Take a cutoff $\eta_{a,R}$ equal to one on $B_R(a)$, supported in
$B_{2R}(a)$, and satisfying the scale-invariant horizontal derivative
bound in \eqref{eq:jl-cutoff-properties}.  Since
\[
 we^{4f}=C_nW^{p+1},
 \qquad w=C_nW^{q_-},
\]
\cref{prop:m0-annular-cutoff} gives
\begin{align*}
 \left(\int\fA_f\eta_{a,R}^8\right)^2
 &\le CR^{-2}
 \left[
  \int_{B_{2R}(a)}W^{p+1}\dx
  +R^{-4}\int_{B_{2R}(a)}W^{q_-}\dx
 \right]\\
 &\qquad\times
 \int_{B_{2R}(a)\setminus B_R(a)}
       \fA_f\eta_{a,R}^8.
\end{align*}
The integral on the left is at least $J_a(R)$, whereas the final
integral is at most $J_a(2R)-J_a(R)$.  This proves
\eqref{eq:annular-defect}.  Left translation of the cutoff proves
uniformity in $a$.
\end{proof}

\subsection{A discrete Riccati lemma}

\begin{lemma}[Discrete Riccati growth]\label{lem:discrete-riccati}
Let $J:[R_0,\infty)\to[0,\infty)$ be nondecreasing and bounded above by a polynomial.  Suppose
\begin{equation}\label{eq:riccati-assumption}
 J(R)^2\le CR^b\bigl(J(2R)-J(R)\bigr)
 \qquad(R\ge R_0).
\end{equation}
If $b>0$, then $J(R)\le C'R^b$.  If $b\le0$, then $J\equiv0$.
\end{lemma}

\begin{proof}
Assume first that $b>0$, fix $R\ge R_0$, and set
\[
 y_k=(2^kR)^{-b}J(2^kR).
\]
Then \eqref{eq:riccati-assumption} implies
\begin{equation}\label{eq:riccati-recurrence}
 y_{k+1}\ge2^{-b}(y_k+cy_k^2).
\end{equation}
Put $a=2^{-b}c$ and choose $A=2/a$.  If $y_0\ge A$, then
$z_k=ay_k$ satisfies
\[
 z_0\ge2,
 \qquad
 z_{k+1}\ge z_k^2.
\]
Inductively, $z_k\ge2^{2^k}$.  On the other hand, the assumed
polynomial upper bound $J(s)\le C_0s^M$ gives
\[
 y_k\le C_0R^{M-b}2^{k(M-b)},
\]
which is at most exponential in $k$.  This contradiction proves
$y_0<A$, uniformly in $R$, and hence $J(R)\le AR^b$.

If $b=0$ and $J(R)>0$ at some $R$, then
\[
 J(2R)\ge J(R)+cJ(R)^2.
\]
Writing $J_k=J(2^kR)$, monotonicity gives
$J_{k+1}-J_k\ge cJ_0^2$, so $J_k\to\infty$.  Once
$cJ_k\ge2$, the quantities $z_k=cJ_k$ obey
$z_{k+1}\ge z_k^2$ and therefore grow double-exponentially, again
contradicting the polynomial upper bound.  Thus $J(R)=0$ for every
$R\ge R_0$.  If $b<0$, then $R^b\le R_0^b$ on this range, reducing
the assertion to the case $b=0$ with a modified constant.
\end{proof}

\subsection{The initial defect dimension}

\begin{lemma}[The lower-power bound]\label{lem:lower-power-initial}
For every $a\in\HH^n$ and $R\ge1$,
\begin{equation}\label{eq:lower-power-initial}
 \int_{B_R(a)}W^{q_-}\dx
 \le
 \begin{cases}
 CR^4,&2\le n\le4,\\
 CR^n,&n>4.
 \end{cases}
\end{equation}
\end{lemma}

\begin{proof}
If $2\le n\le4$, then $q_-\le p$.  H\"older's inequality, \cref{prop:critical-morrey}, and $|B_R|\asymp R^Q$ give
\[
 \int_{B_R(a)}W^{q_-}
 \le\left(\int_{B_R(a)}W^p\right)^{q_-/p}
 |B_R|^{1-q_-/p}
 \le CR^4,
\]
because
\[
 n\frac{q_-}{p}+Q\left(1-\frac{q_-}{p}\right)
 =Q-(Q-n)\frac{q_-}{p}
 =2n+2-(n+2)\frac{2(n-1)}{n+2}=4.
\]
If $n>4$, then $q_->p$, and boundedness gives
\[
 \int_{B_R(a)}W^{q_-}
 \le\|W\|_\infty^{q_--p}\int_{B_R(a)}W^p
 \le CR^n.
\]
\end{proof}

\begin{proposition}[Initial codimension-two budget]\label{prop:initial-defect-growth}
For a bounded solution $W$,
\begin{equation}\label{eq:initial-energy-growth}
 \sup_a\int_{B_R(a)}W^{p+1}\dx\le CR^n.
\end{equation}
Moreover,
\begin{equation}\label{eq:initial-defect-growth}
 n=2:\ \fA_f\equiv0;
 \qquad
 n\ge3:\ \sup_aJ_a(R)\le CR^{n-2}.
\end{equation}
\end{proposition}

\begin{proof}
Boundedness and \cref{prop:critical-morrey} give
\[
 \int_{B_R(a)}W^{p+1}
 \le\|W\|_\infty\int_{B_R(a)}W^p
 \le CR^n.
\]
Insert this and \eqref{eq:lower-power-initial} into \eqref{eq:annular-defect}.  The term containing $W^{q_-}$ is lower order, and therefore
\[
 J_a(R)^2\le CR^{n-2}\bigl(J_a(2R)-J_a(R)\bigr).
\]
The polynomial upper bound required in \cref{lem:discrete-riccati} follows from \cref{cor:defect-pointwise,lem:lower-power-initial}.  Apply that lemma with $b=n-2$.
\end{proof}

\section{Defect quantization and height-level covering}\label{sec:defect-quantization}

Let $W$ be a fixed bounded exact-Riesz/Morrey solution of \eqref{eq:main-equation}.  Unless stated otherwise, we assume that $W$ is not a bubble.

\subsection{Scaling of the defect}

For $a\in\HH^n$, define the natural scale
\begin{equation}\label{eq:mu-a}
 \mu(a)=W(a)^{-1/n}
\end{equation}
and the normalized solution
\begin{equation}\label{eq:Va-def}
 V_a(\zeta)=\mu(a)^nW(a\delta_{\mu(a)}\zeta).
\end{equation}
Then $V_a(e)=1$.  We call \eqref{eq:Va-def} $L$-good if
\begin{equation}\label{eq:L-good}
 0<V_a\le2^n\quad\text{on }B_L(e).
\end{equation}

\begin{lemma}[Defect scaling]\label{lem:defect-scaling}
For $\mu=\mu(a)$,
\begin{equation}\label{eq:defect-density-scaling}
 \fA_{f_{V_a}}(\zeta)
 =\mu^{Q+2}\fA_{f_W}(a\delta_\mu\zeta).
\end{equation}
Consequently,
\begin{equation}\label{eq:defect-integral-scaling}
 \int_{B_L(e)}\fA_{f_{V_a}}\,d\zeta
 =\mu^2\int_{B_{L\mu}(a)}\fA_{f_W}\dx.
\end{equation}
\end{lemma}

\begin{proof}
The logarithmic factors satisfy
\[
 f_{V_a}(\zeta)=f_W(a\delta_\mu\zeta)+\log\mu.
\]
A horizontal derivative has weight one and $\partial_t$ has weight two.  Substitution in \eqref{eq:DE-tensors}--\eqref{eq:defect-density} shows that every summand has total weight $Q+2$, proving \eqref{eq:defect-density-scaling}.  Haar measure contributes $\mu^{-Q}$ after the change of variables, which gives \eqref{eq:defect-integral-scaling}.
\end{proof}

\subsection{A uniform natural-scale quantum}

\begin{proposition}[Uniform defect quantum]\label{prop:defect-quantum}
There exist $L_0<\infty$ and $\varepsilon_0>0$, depending on the fixed solution $W$, such that every $L_0$-good normalization at $a$ satisfies
\begin{equation}\label{eq:defect-quantum}
 \int_{B_{L_0\mu(a)}(a)}\fA_{f_W}\dx
 \ge\varepsilon_0\mu(a)^{-2}.
\end{equation}
\end{proposition}

\begin{proof}
Suppose that no such pair $(L_0,\varepsilon_0)$ exists.  For each integer $j$ there is a point $a_j$ such that the natural normalization $V_j=V_{a_j}$ is $j$-good and
\begin{equation}\label{eq:vanishing-normalized-defect}
 \int_{B_j(e)}\fA_{f_{V_j}}<\frac1j.
\end{equation}
The sequence satisfies $V_j(e)=1$ and $V_j\le2^n$ on expanding balls.  By \cref{lem:expanding-compactness},
\[
 V_j\longrightarrow V_\infty
 \quad\text{in }C^\infty_{H,\loc}(\HH^n)
\]
after taking a subsequence.  Affine invariance preserves the uniform Morrey constant, and \cref{prop:riesz-compactness} passes the exact global representation to $V_\infty$.  From \eqref{eq:vanishing-normalized-defect} and local smooth convergence,
\[
 \fA_{f_{V_\infty}}\equiv0.
\]
Thus \cref{prop:zero-defect-rigidity} implies $V_\infty\in\cM_{\rm bub}$.

Every $V_j$ is an affine transform of the same non-bubble $W$ and is therefore a non-bubble.  The sequence $(V_j)$ is a uniform exact-Riesz/Morrey family converging locally to a bubble, contradicting \cref{thm:intro-isolation}.  This proves the proposition.
\end{proof}

\begin{remark}\label{rem:quantum-quantifiers}
The expanding radius in the contradiction is essential.  A fixed finite $L$-good window need not be close to a bubble.  The failure of a uniform quantum permits the simultaneous choice $L_j\to\infty$ and normalized defect tending to zero.
\end{remark}

\subsection{Packing a fixed height layer}

Assume that for some $d\in[0,n-2]$,
\begin{equation}\label{eq:defect-growth-d}
 \sup_{a\in\HH^n}J_a(R)\le C_JR^d
 \qquad(R\ge1).
\end{equation}
Fix $a_0\in\HH^n$, $R\ge1$, and a dyadic value $\mu\le R$.  Define
\begin{equation}\label{eq:height-layer}
 E_\mu
 =\left\{x\in B_R(a_0):
 \mu\le W(x)^{-1/n}<2\mu
 \right\}.
\end{equation}

\begin{lemma}[Same-level packing]\label{lem:same-level-packing}
Under \eqref{eq:defect-growth-d},
\begin{equation}\label{eq:level-volume}
 |E_\mu|\le CR^d\mu^{Q+2}
 \qquad(\mu\le R).
\end{equation}
\end{lemma}

\begin{proof}
Choose a maximal $A_*L_0\mu$-separated set $\{x_i\}_{i=1}^N\subset E_\mu$, where $A_*>12$ is fixed.  The maximality implies that the balls $B_{A_*L_0\mu}(x_i)$ cover $E_\mu$.

Apply \cref{lem:doubling} on the complete space $\HH^n$ to $M=W^{1/n}$, starting at each $x_i$ with $k=L_0$.  We obtain $y_i$ satisfying
\begin{equation}\label{eq:level-displacement}
 M(y_i)\ge M(x_i),
 \qquad
 d(x_i,y_i)\le\frac{2L_0}{M(x_i)}\le4L_0\mu.
\end{equation}
If $\nu_i=M(y_i)^{-1}$, then the normalization at $y_i$ is $L_0$-good and
\begin{equation}\label{eq:nu-upper}
 \nu_i\le M(x_i)^{-1}<2\mu.
\end{equation}
The separation of the $x_i$, \eqref{eq:level-displacement}, and $A_*>12$ show that the balls $B_{L_0\nu_i}(y_i)$ are pairwise disjoint.

By \cref{prop:defect-quantum} and \eqref{eq:nu-upper}, each such ball
carries at least $c\mu^{-2}$ defect.  Since $\mu\le R$, all these
balls lie in $B_{C_{L_0}R}(a_0)$.  Therefore
\eqref{eq:defect-growth-d} gives
\begin{equation}\label{eq:number-balls}
 Nc\mu^{-2}
 \le\sum_{i=1}^N\int_{B_{L_0\nu_i}(y_i)}\fA_f
 \le J_{a_0}(C_{L_0}R)
 \le CR^d.
\end{equation}
Thus $N\le CR^d\mu^2$.  The covering property and volume growth now yield
\[
 |E_\mu|
 \le NC(A_*L_0\mu)^Q
 \le CR^d\mu^{Q+2}.
\]
\end{proof}

\subsection{Interpolation of the source and defect budgets}

On $E_\mu$ one has $W\le\mu^{-n}$.  The Morrey source budget gives
\begin{equation}\label{eq:source-layer-budget}
 \int_{E_\mu}W^{p+1}\dx
 \le\mu^{-n}\int_{B_R(a_0)}W^p\dx
 \le CR^n\mu^{-n}.
\end{equation}
On the other hand, $n(p+1)=Q$, and \cref{lem:same-level-packing} gives
\begin{equation}\label{eq:defect-layer-budget}
 \int_{E_\mu}W^{p+1}\dx
 \le\mu^{-Q}|E_\mu|
 \le CR^d\mu^2.
\end{equation}
Combining the two,
\begin{equation}\label{eq:two-layer-budgets}
 \int_{E_\mu}W^{p+1}\dx
 \le C\min\{R^d\mu^2,R^n\mu^{-n}\}
 \qquad(\mu\le R).
\end{equation}

\begin{proposition}[Energy improvement]\label{prop:energy-improvement}
If \eqref{eq:defect-growth-d} holds for some $d\in[0,n-2]$, then
\begin{equation}\label{eq:energy-improvement}
 \sup_{a_0\in\HH^n}\int_{B_R(a_0)}W^{p+1}\dx
 \le CR^{\alpha(d)},
 \qquad
 \alpha(d)=\frac{n(d+2)}{n+2}.
\end{equation}
\end{proposition}

\begin{proof}
Partition the range of $W^{-1/n}$ into dyadic layers.  Balance the two terms in \eqref{eq:two-layer-budgets} at
\begin{equation}\label{eq:balancing-scale}
 \mu_*^{n+2}=R^{n-d}.
\end{equation}
Because $0\le d\le n-2$ and $R\ge1$, one has $1\le\mu_*\le R$ up to a harmless adjustment of the dyadic origin.  Sum the first term in \eqref{eq:two-layer-budgets} over $\mu\le\mu_*$ and the second over $\mu_*\le\mu\le R$.  Both sums are geometric and are bounded by
\[
 CR^d\mu_*^2
 =CR^{n(d+2)/(n+2)}.
\]
The remaining low-value region $\{W^{-1/n}>R\}$ satisfies $W\le R^{-n}$, and hence
\begin{equation}\label{eq:low-value-tail}
 \int_{B_R(a_0)\cap\{W^{-1/n}>R\}}W^{p+1}
 \le R^{-n}\int_{B_R(a_0)}W^p
 \le C.
\end{equation}
This proves \eqref{eq:energy-improvement}, uniformly in $a_0$.
\end{proof}

\begin{remark}
Defect balls are required to be disjoint only within a fixed height layer.  The same defect may be counted again at another height; no cross-layer Carleson packing is needed because the two dyadic sums in the proof of \cref{prop:energy-improvement} are geometric.
\end{remark}

\section{Finite defect-dimension descent and classification}\label{sec:dimension-descent}

We now iterate the two budgets.  The argument takes place for one fixed bounded exact-Riesz/Morrey solution $W$.

\subsection{Defect dimension from energy dimension}

For $D>0$, let $\mathbf E(D)$ denote the all-centre energy bound
\begin{equation}\label{eq:ED}
 \sup_{a\in\HH^n}\int_{B_R(a)}W^{p+1}\dx
 \le C_DR^D
 \qquad(R\ge1).
\end{equation}

\begin{lemma}[Lower power under $\mathbf E(D)$]\label{lem:lower-power-D}
Assume \eqref{eq:ED}.  Then
\begin{equation}\label{eq:lower-power-D}
 \sup_a\int_{B_R(a)}W^{q_-}\dx
 \le
 \begin{cases}
 CR^4,&2\le n\le4,\\
 CR^{4+(n-4)D/n},&n>4.
 \end{cases}
\end{equation}
\end{lemma}

\begin{proof}
For $n\le4$, use the proof of \cref{lem:lower-power-initial}.  If $n>4$, the exact identity
\begin{equation}\label{eq:q-interpolation}
 q_-=\frac4n p+\frac{n-4}{n}(p+1)
\end{equation}
and H\"older's inequality give
\begin{align*}
 \int_{B_R(a)}W^{q_-}
 &\le
 \left(\int_{B_R(a)}W^p\right)^{4/n}
 \left(\int_{B_R(a)}W^{p+1}\right)^{(n-4)/n}\\
 &\le CR^{4+(n-4)D/n}.
\end{align*}
\end{proof}

\begin{proposition}[Defect-dimension reduction]\label{prop:defect-dimension-reduction}
Assume $0<D\le n$ and $\mathbf E(D)$.  Then
\begin{equation}\label{eq:defect-dimension-reduction}
 D\le2\ \Longrightarrow\ \fA_f\equiv0;
 \qquad
 D>2\ \Longrightarrow\ \sup_aJ_a(R)\le CR^{D-2}.
\end{equation}
\end{proposition}

\begin{proof}
Insert \eqref{eq:ED} and \eqref{eq:lower-power-D} into the annular estimate \eqref{eq:annular-defect}.  The contribution of the energy term is
\[
 CR^{D-2}\bigl(J_a(2R)-J_a(R)\bigr).
\]
If $n\le4$, the lower-power contribution has coefficient $CR^{-2}$ and is no larger.  If $n>4$, its coefficient has exponent
\begin{equation}\label{eq:lower-power-exponent-comparison}
 -2+\frac{n-4}{n}D
 =(D-2)-\frac{4D}{n}<D-2.
\end{equation}
Therefore
\begin{equation}\label{eq:riccati-D}
 J_a(R)^2
 \le CR^{D-2}\bigl(J_a(2R)-J_a(R)\bigr).
\end{equation}
The polynomial upper bound for $J_a$ follows from \cref{cor:defect-pointwise,lem:lower-power-D}.  Apply \cref{lem:discrete-riccati} with $b=D-2$, uniformly in the centre.
\end{proof}

\subsection{Contraction of the energy dimension}

\begin{corollary}[Energy contraction]\label{cor:energy-contraction}
Suppose $2<D\le n$, $W$ is not a bubble, and $\mathbf E(D)$ holds.  Then
\begin{equation}\label{eq:energy-contraction}
 \mathbf E\left(\frac{n}{n+2}D\right)
\end{equation}
holds.
\end{corollary}

\begin{proof}
By \cref{prop:defect-dimension-reduction}, \eqref{eq:defect-growth-d} holds with $d=D-2$.  Since $D\le n$, one has $0<d\le n-2$.  Apply \cref{prop:energy-improvement}:
\[
 \alpha(D-2)
 =\frac{nD}{n+2}.
\]
\end{proof}

\begin{theorem}[Bounded-profile Liouville theorem]\label{thm:bounded-profile}
Let $n\ge2$.  Every bounded positive exact-Riesz/Morrey solution of \eqref{eq:main-equation} is a Jerison--Lee bubble.
\end{theorem}

\begin{proof}
Assume that $W$ is not a bubble.  Boundedness and the source Morrey estimate give $\mathbf E(D_0)$ with $D_0=n$; see \eqref{eq:initial-energy-growth}.  Define
\begin{equation}\label{eq:D-iteration}
 D_{k+1}=\frac{n}{n+2}D_k,
 \qquad
 D_k=n\left(\frac{n}{n+2}\right)^k.
\end{equation}
As long as $D_k>2$, \cref{cor:energy-contraction} proves $\mathbf E(D_{k+1})$.  For every fixed $n$, a finite $k$ satisfies $D_k\le2$.  Then \cref{prop:defect-dimension-reduction} gives $\fA_f\equiv0$, and \cref{prop:zero-defect-rigidity} makes $W$ a bubble, a contradiction.
\end{proof}

\subsection{Excluding unbounded entire solutions}

\begin{proposition}\label{prop:no-unbounded}
Let $n\ge2$.  Every positive entire solution of \eqref{eq:main-equation} is bounded.
\end{proposition}

\begin{proof}
Suppose $u$ is unbounded and put $M=u^{1/n}$.  Choose $x_j$ with $M(x_j)\to\infty$ and apply \cref{lem:doubling} with $k=j$.  This gives $a_j$ and $\mu_j=M(a_j)^{-1}$ such that
\begin{equation}\label{eq:unbounded-rescaling}
 W_j(\zeta)=\mu_j^nu(a_j\delta_{\mu_j}\zeta),
 \qquad
 W_j(e)=1,
 \qquad
 W_j\le2^n\quad\text{on }B_j(e).
\end{equation}
By \cref{thm:exact-riesz,prop:critical-morrey}, the sequence is a uniform exact-Riesz/Morrey family.  Expanding-window compactness and \cref{prop:riesz-compactness} give
\[
 W_j\longrightarrow W_\infty
 \quad\text{in }C^\infty_{H,\loc}(\HH^n),
\]
where $W_\infty$ is a bounded positive exact-Riesz/Morrey solution.  By \cref{thm:bounded-profile}, $W_\infty$ is a bubble.

If some $W_j$ is a bubble, affine invariance makes $u$ a bubble and hence bounded, contrary to assumption.  Thus every $W_j$ is a non-bubble.  This contradicts the sequential isolation theorem \cref{thm:intro-isolation}.
\end{proof}

\begin{proof}[Proof of \cref{thm:main}]
By \cref{thm:exact-riesz,prop:critical-morrey}, every positive entire solution is exact-Riesz/Morrey.  By \cref{prop:no-unbounded} it is bounded.  Apply \cref{thm:bounded-profile}.
\end{proof}

\begin{proof}[Proof of \cref{cor:all-dimensional}]
For $n\ge2$, use \cref{thm:main}.  The case $n=1$ is the unconditional Liouville theorem of Catino, Li, Monticelli, and Roncoroni \cite{CatinoLiMonticelliRoncoroni2025}.
\end{proof}

\appendix
\section{Conformal normalization and zero-defect rigidity}\label{app:normalizations}

This appendix fixes the constants that connect the Heisenberg equation, the CR sphere, and the Flynn--V\'etois form of the Jerison--Lee identity.

\subsection{Real and complex sub-Laplacians}

With the group law \eqref{eq:group-law},
\[
 Z_\alpha=\frac12(X_\alpha-\ii Y_\alpha)
 =\partial_{z_\alpha}+\ii\overline z_\alpha\partial_t,
 \qquad
 [Z_\alpha,Z_{\bar\beta}]=-2\ii\delta_{\alpha\beta}\partial_t.
\]
For the standard contact form, $\Delta_bu=-\Rea(u_{\alpha\bar\alpha})$, and a direct computation gives
\begin{equation}\label{eq:laplacian-dictionary}
 4\Delta_b=-\sum_{j=1}^n(X_j^2+Y_j^2)=L.
\end{equation}
Thus the equation used throughout the paper is exactly the Flynn--V\'etois normalization
\[
 4\Delta_bu=n^2u^{1+2/n}.
\]
If instead one starts from the coefficient-one equation $Lv=v^p$, then
\[
 u=n^{-n}v,
 \qquad
 U_0^{(v)}=(2n)^n\bigl((1+|z|^2)^2+t^2\bigr)^{-n/2},
\]
whereas the present paper uses
\[
 U_0=n^{-n}U_0^{(v)}
 =2^n\bigl((1+|z|^2)^2+t^2\bigr)^{-n/2}.
\]

\subsection{Cayley measure and spherical spectrum}

For the group convention \eqref{eq:group-law}, we take
\begin{equation}\label{eq:cayley-fixed}
 \cC(z,t)
 =\left(
 \frac{2z}{1+|z|^2+\ii t},
 \frac{1-|z|^2-\ii t}{1+|z|^2+\ii t}
 \right).
\end{equation}
It omits $N=(0,\ldots,0,-1)$.  We define the spherical measure used in the paper by the exact pullback identity
\begin{equation}\label{eq:spherical-measure-normalization}
 \cC^*(dV_S)=U_0^{p+1}\dx.
\end{equation}
If $d\sigma$ denotes Euclidean surface measure on the unit sphere, a direct Jacobian computation gives $dV_S=2\,d\sigma$.  The choice removes an otherwise irrelevant factor from every reciprocity formula.
More explicitly,
\[
 J_{\cC}
 =2^{2n+1}\bigl((1+|z|^2)^2+t^2\bigr)^{-(n+1)},
 \qquad
 U_0^{p+1}=2J_{\cC}.
\]
If $\Omega_{\rm FL}=J_{\cC}^{1/(p+1)}$ denotes the natural Cayley factor for Euclidean surface measure, then
\[
 \Omega_{\rm FL}=2^{-n/[2(n+1)]}U_0.
\]

Let $\cL_{S,\mathrm{FL}}$ be the standard spherical conformal sub-Laplacian in the Frank--Lieb normalization.  It is a globally smooth, positive, self-adjoint, invertible operator on the sphere.  In the present normalization,
\begin{equation}\label{eq:sphere-operator-dictionary}
 \cL_S=4\cL_{S,\mathrm{FL}}.
\end{equation}
Equivalently, $\cL_S$ is characterized by the covariance relation \eqref{eq:conformal-covariance}.  Since $LU_0=n^2U_0^p$, this gives
\[
 \kappa_S=\cL_S1=n^2.
\]
The bispherical harmonics $\cH_{r,s}$ diagonalize $\cL_S$ and
\begin{equation}\label{eq:spectrum-appendix}
 \left.\cL_S\right|_{\cH_{r,s}}
 =4\left(r+\frac n2\right)\left(s+\frac n2\right).
\end{equation}
Indeed, the horizontal spherical sub-Laplacian in the Frank--Lieb normalization has eigenvalue
$rs+\frac n2(r+s)$ on $\cH_{r,s}$ \cite[Section~5]{FrankLieb2012}; adding the conformal constant and multiplying by four gives \eqref{eq:spectrum-appendix}.  This also proves the kernel statement \eqref{eq:sphere-kernel}.

Let $G_S$ denote the inverse kernel of $\cL_S$ with respect to $dV_S$.  Applying the linear fundamental-solution identity on $\HH^n$ to a smooth spherical source, using \eqref{eq:conformal-covariance}, and then changing variables by \eqref{eq:spherical-measure-normalization} gives the exact identity
\begin{equation}\label{eq:green-dictionary-appendix}
 G_S(\cC\xi,\cC\eta)
 =\frac{\GammaH(\eta^{-1}\xi)}{U_0(\xi)U_0(\eta)}.
\end{equation}
Thus no undetermined dimensional constant occurs in \eqref{eq:green-dictionary}.
If $G_{S,\mathrm{FL}}$ is instead the inverse kernel of $\cL_{S,\mathrm{FL}}$ with respect to $d\sigma$, then $G_S=G_{S,\mathrm{FL}}/8$; both the factor four in the operator and the factor two in the measure contribute.

\subsection{The local rigidity input}

\begin{proposition}[Zero-defect rigidity]\label{prop:zero-defect-rigidity}
Let $n\ge2$ and let $u>0$ be a smooth entire solution of
\[
 4\Delta_bu=n^2u^{1+2/n}.
\]
Define $f$, $D_{\alpha\beta}$, $E_{\alpha\bar\beta}$, $D_\alpha$, $E_\alpha$, $G_\alpha$, and $\fA_f$ by \eqref{eq:f-def}--\eqref{eq:defect-density}.  If $\fA_f\equiv0$, then there are $\mu\in\CC^n$ and $\lambda\in\CC$ with $|\mu|^2<4\Ima\lambda$ such that
\begin{equation}\label{eq:JL-explicit-family}
 u(z,t)
 =\frac{(4\Ima\lambda-|\mu|^2)^{n/2}}
 {\bigl|t+\ii|z|^2+\mu\mathbin{\cdot}z+\lambda\bigr|^n},
\end{equation}
where $\mu\mathbin{\cdot}z=\sum_{\alpha=1}^n\mu_\alpha z_\alpha$ is complex bilinear.
Equivalently, $u\in\cM_{\rm bub}$.
\end{proposition}

\begin{proof}
The positive weights in \eqref{eq:defect-density} show that $\fA_f\equiv0$ gives, in particular,
\[
 D_{\alpha\beta}=0,
 \qquad
 E_{\alpha\bar\beta}=0.
\]
Since $n\ge2$, the second identity is the local CR-pluriharmonicity condition
\[
 f_{\alpha\bar\beta}
 =\frac1n f_{\gamma\bar\gamma}\delta_{\alpha\bar\beta}.
\]
The standard characterization of CR-pluriharmonic functions \cite{Bedford1980}, applied in exactly this form by Jerison and Lee \cite[p.~12, proof of Corollary~4.2]{JerisonLee1988}, and the simple connectedness of $\HH^n$ give a global CR-holomorphic function $h$ with $\Rea h=f$.  Hence $h_\alpha=2f_\alpha$.  Set
\[
 \psi=e^{-h}.
\]
Then $\psi$ is nonvanishing and CR holomorphic, $|\psi|=e^{-f}$, and
\[
 \psi_{\alpha\beta}
 =\bigl(-2f_{\alpha\beta}+4f_\alpha f_\beta\bigr)\psi
 =-2D_{\alpha\beta}\psi=0.
\]

We now integrate this last system directly on $\HH^n$.  Write
\[
 w=t+\ii|z|^2;
\]
then $w$ is CR holomorphic and $w_0=1$.  Fixing $\alpha$ and using
$[Z_\alpha,Z_{\bar\alpha}]=-2\ii\partial_t$, CR holomorphicity and
$\psi_{\alpha\alpha}=0$ give
\[
 0=Z_{\bar\alpha}(\psi_{\alpha\alpha})
   =4\ii\,\partial_tZ_\alpha\psi.
\]
Thus $Z_\alpha\psi_0=0$, while CR holomorphicity gives $Z_{\bar\alpha}\psi_0=0$.  Hence every horizontal derivative of $\psi_0$ vanishes, and bracket generation makes $\psi_0=a$ a complex constant.  The function $k=\psi-aw$ is CR holomorphic and independent of $t$.  It is therefore holomorphic in $z$; moreover $k_{\alpha\beta}=0$, so
\[
 \psi(z,t)=a\bigl(t+\ii|z|^2\bigr)+b\mathbin{\cdot}z+c
\]
for constants $a,c\in\CC$ and $b\in\CC^n$.

If $a=0$, nonvanishing forces $b=0$, while the equation for $f$ excludes a constant $\psi$.  Thus $a\ne0$.  Put $\mu=b/a$ and $\lambda=c/a$.  The polynomial
\[
 t+\ii|z|^2+\mu\mathbin{\cdot}z+\lambda
\]
has no zero on $\HH^n$ precisely when
$4\Ima\lambda-|\mu|^2>0$.  Since $u=2^ne^{nf}=2^n|\psi|^{-n}$, substitution in $4\Delta_bu=n^2u^{1+2/n}$ fixes the remaining modulus:
\[
 |a|^{-2}=\Ima\lambda-\frac{|\mu|^2}{4},
 \qquad
 2^n|a|^{-n}=(4\Ima\lambda-|\mu|^2)^{n/2}.
\]
This proves \eqref{eq:JL-explicit-family}.  The last family agrees with the affine orbit \eqref{eq:bubble-family-intro} by completing the square and applying a left translation and a dilation.
\end{proof}

\begin{remark}[Rigidity input boundary]\label{rem:rigidity-verification-boundary}
The only external fact in the preceding proof is Bedford's local characterization of CR-pluriharmonic functions \cite{Bedford1980}, used after $E_{\alpha\bar\beta}=0$; Jerison and Lee invoke it at precisely the same point on p.~12 of \cite{JerisonLee1988}.  Everything after the construction of $h$ is written above on the whole Heisenberg group.  In particular, the rigidity step uses neither $u\in L^{p+1}$ nor any decay at infinity.
\end{remark}

\section{Standard subelliptic analytic inputs}\label{app:analytic-inputs}

For clarity, we list the local and potential-theoretic tools used in the paper.  All constants below are invariant under left translation; after scaling, their dependence is only on the displayed radius ratios, the dimension, and coefficient bounds.

\subsection{Harnack and interior estimates}

Let $v>0$ solve
\[
 Lv=a(\xi)v
\]
on $B_1$, with $a\in L^\infty$.  The scale-invariant Harnack inequality gives
\[
 \sup_{B_{1/2}}v
 \le C(n,\|a\|_\infty)\inf_{B_{1/2}}v.
\]
For H\"ormander sums of squares, the maximum-principle and Harnack
framework goes back to Bony \cite{Bony1969}; the version with a bounded
potential used here follows from the scale-invariant estimates of
Citti--Garofalo--Lanconelli \cite{CittiGarofaloLanconelli1993}.
Interior estimates for H\"ormander sums of squares then give, for nested balls $B_r\Subset B_R$ and every integer $k$,
\[
 \|v\|_{S_H^{k+2,q}(B_r)}
 \le C\bigl(\|Lv\|_{S_H^{k,q}(B_R)}+\|v\|_{L^q(B_R)}\bigr).
\]
Together with the subelliptic Sobolev embeddings, this yields the
Schauder-type bootstrap used in
\cref{lem:expanding-compactness,lem:relative-derivatives}.  The
function-space and interior-estimate inputs are part of the
Folland--Stein theory; see \cite{FollandStein1974,FollandStein1982}.

On compact subsets of the CR sphere the same estimates hold in pseudohermitian coordinate charts.  In particular, bounded subsets of $S_H^{2,2}(K')$ are precompact in $L^2(K)$ when $K\Subset K'$.  This is the local Rellich input used in \cref{lem:normalized-defect}.

\subsection{Positive harmonic Liouville theorem}

If $h\ge0$ and $Lh=0$ on all of $\HH^n$, then $h$ is constant.  One proof is the real-sub-Laplacian Liouville theorem of Bonfiglioli and Lanconelli \cite{BonfiglioliLanconelli2001}.  Alternatively, let $m=\inf h$, choose $y_k$ with $h(y_k)-m\to0$, and apply the scale-invariant Harnack inequality to $h-m\ge0$ on a ball whose concentric half-ball contains both a fixed point $x$ and $y_k$.  The uniform Harnack constant gives $h(x)-m\le C(h(y_k)-m)\to0$, so $h\equiv m$.

\subsection{Green exhaustion}

For a regular bounded exhaustion $\Omega_R\nearrow\HH^n$, the Friedrichs Dirichlet forms increase to the full-space form.  Their killed heat kernels increase pointwise to the Heisenberg heat kernel, and hence the integrated Green kernels satisfy
\[
 G_R(\xi,\eta)\uparrow\GammaH(\eta^{-1}\xi).
\]
This formulation avoids imposing classical boundary regularity at characteristic points.  It is the only Dirichlet-form convergence input in the proof of \cref{thm:exact-riesz}.  Systematic accounts of the potential theory for sub-Laplacians on stratified groups, including Green functions, the Poisson kernel and the Liouville property, are given in \cite{BonfiglioliLanconelliUguzzoni2007}.

\section{The \texorpdfstring{$m=0$}{m=0} cutoff estimate from the
Jerison--Lee identity}
\label{app:jl-cutoff}

This appendix proves the annular estimate used in
\cref{prop:annular-defect}, starting from the $m=0$ Jerison--Lee
divergence identity.  The pointwise identity itself is the algebraic
input from \cite[Proposition~4.1, formula~(4.2)]{JerisonLee1988}; a
direct verification of that identity in coordinates is given by Ma and
Ou \cite[Proposition~2.1]{MaOu2023}.  All localization and lower-order
estimates needed after that identity are proved below.  In particular, no Liouville theorem or growth assumption
from \cite{FlynnVetois2023} is used.  We retain the notation of
\cref{sec:jl-defect} and put
\[
 w=e^{2(n-1)f}.
\]
All integrations below are over $\HH^n$ unless a domain is displayed.

\subsection{The pointwise identity and its coercivity}

We first translate the formula of Jerison and Lee into the conventions
of this paper.  Their standard holomorphic frame and component
derivatives are our $Z_\alpha$ and subscript derivatives.  Their
equation (4.1) is
\[
 \Rea(f_{\alpha\bar\alpha})
   +n f_\alpha f_{\bar\alpha}+ne^{2f}=0,
\]
which is exactly \eqref{eq:f-equation}, because
$\Delta_bf=-\Rea(f_{\alpha\bar\alpha})$.  The tensors
\[
 D_{\alpha\beta},\quad E_{\alpha\bar\beta},\quad
 D_\alpha,\quad E_\alpha
\]
in their formula agree with
\eqref{eq:DE-tensors}--\eqref{eq:DE-one-forms}; in particular, the
component contractions defining $D_\alpha$ and $E_\alpha$ are
unchanged.  Their remaining one-form is
\[
 \ii f_{0\alpha}-\ii f_0f_\alpha+e^{2f}f_\alpha
       +f_\beta f_{\bar\beta}f_\alpha
 =\ii f_{0\alpha}+gf_\alpha=G_\alpha.
\]
Define the complex horizontal vector field
\begin{equation}\label{eq:jl-flux}
 \mathcal V_\alpha
 =g(D_\alpha+E_\alpha)
  -\ii f_0(D_\alpha-3E_\alpha+3G_\alpha).
\end{equation}
After grouping the vector field inside the divergence in formula (4.2)
of Jerison and Lee, it is precisely $\mathcal V_\alpha$ in
\eqref{eq:jl-flux}.  Thus that formula becomes
\begin{align}
 \Rea\bigl((w\mathcal V_\alpha)_{\bar\alpha}\bigr)
 =w\bigl[&e^{2f}(|D_{\alpha\beta}|^2
                    +|E_{\alpha\bar\beta}|^2)
          +|G_\alpha|^2+|G_\alpha+D_\alpha|^2
          +|G_\alpha-E_\alpha|^2 \notag\\
 &\quad
          +|D_{\alpha\beta}f_{\bar\gamma}
                    +E_{\alpha\bar\gamma}f_\beta|^2\bigr].
 \label{eq:jl-pointwise}
\end{align}
The last tensor square is nonnegative and will not be used separately.
The three completed one-form squares in
\eqref{eq:jl-pointwise} satisfy
\begin{equation}\label{eq:one-form-coercivity}
 |G|^2+|G+D|^2+|G-E|^2
 \ge c\bigl(|D|^2+|E|^2+|G|^2\bigr)
\end{equation}
for a numerical $c>0$.  Thus the right-hand side of
\eqref{eq:jl-pointwise} dominates $c\fA_f$.

\begin{lemma}[The basic cutoff inequality]\label{lem:jl-basic-cutoff}
Let $\phi\ge0$ be smooth and compactly supported.  Then
\begin{align}
 \int\fA_f\phi
 &\le C
 \left(\int w(e^{4f}+|\partial f|^4+f_0^2)
                 \phi^{-1}|\partial\phi|^2\right)^{1/2}
 \left(\int_{\{\partial\phi\ne0\}}\fA_f\phi\right)^{1/2}.
 \label{eq:jl-basic-cutoff}
\end{align}
The integrand containing $\phi^{-1}$ is understood as zero where
$\phi=|\partial\phi|=0$, or equivalently by a positive
regularization followed by monotone passage to the limit.  For the
power cutoff used below, the product extends continuously by zero.
\end{lemma}

\begin{proof}
Multiply \eqref{eq:jl-pointwise} by $\phi$ and integrate by parts.
By \eqref{eq:one-form-coercivity}, the resulting left-hand side is at
least $c\int\fA_f\phi$.  Since
\[
 |g|\le e^{2f}+|\partial f|^2+|f_0|,
\]
the boundary term is bounded pointwise by
\[
 Cw(e^{2f}+|\partial f|^2+|f_0|)
       (|D|+|E|+|G|)|\partial\phi|.
\]
Cauchy--Schwarz, restricted in the second factor to
$\{\partial\phi\ne0\}$, gives \eqref{eq:jl-basic-cutoff}.
\end{proof}

\subsection{The two lower-order estimates}

The following two estimates are stated with the weights needed for the
cutoff argument.  Their proofs are included to keep the exponent
bookkeeping explicit.  Negative powers at the zero set are understood
in the extended-value sense, or by positive regularization whenever
the right-hand side is finite.  In the only application
$\phi=\eta^6$, $\sigma=1/3$, and all weighted products displayed
below extend continuously by zero.

\begin{lemma}[Control of the Reeb derivative]\label{lem:jl-f0-control}
For every smooth compactly supported $\phi\ge0$, every
$\sigma\in\RR$, and every $\varepsilon>0$,
\begin{align}
 \int wf_0^2\phi
 \le C\int w\bigl[&e^{4f}\phi+|\partial f|^4\phi
       +\phi^{-3}|\partial\phi|^4
       +\varepsilon^{-2}\phi^{1-2\sigma}\notag\\
 &\quad+\varepsilon(|D|^2+|E|^2+|G|^2)
                      \phi^{1+\sigma}\bigr].
 \label{eq:jl-f0-control}
\end{align}
\end{lemma}

\begin{proof}
Using $f_{0\alpha}=-\ii G_\alpha+\ii gf_\alpha$ and
\eqref{eq:f-equation}, a direct differentiation gives
\begin{align}
 \Ima\bigl((wf_0f_\alpha\phi)_{\bar\alpha}\bigr)
 ={}&w(nf_0^2-e^{2f}|\partial f|^2-|\partial f|^4)\phi\notag\\
 &+\Ima\bigl(wf_\alpha(\ii\overline{G_\alpha}\phi
                         +f_0\phi_{\bar\alpha})\bigr).
 \label{eq:jl-f0-divergence}
\end{align}
After integration, the left-hand side vanishes.  Young's inequality
gives
\[
 e^{2f}|\partial f|^2
 \le\tfrac12(e^{4f}+|\partial f|^4)
\]
and
\[
 |f_0||\partial f||\partial\phi|
 \le\tfrac14f_0^2\phi
   +C\bigl(|\partial f|^4\phi
              +\phi^{-3}|\partial\phi|^4\bigr).
\]
Finally, the weighted three-factor Young inequality yields
\begin{align*}
 |\partial f|(|D|+|E|+|G|)\phi
 \le C\bigl[&|\partial f|^4\phi
      +\varepsilon^{-2}\phi^{1-2\sigma}\\
 &+\varepsilon(|D|^2+|E|^2+|G|^2)
                    \phi^{1+\sigma}\bigr].
\end{align*}
Absorbing the first term containing $f_0^2$ proves
\eqref{eq:jl-f0-control}.
\end{proof}

\begin{lemma}[Control of the horizontal gradient]\label{lem:jl-gradient-control}
Assume $n\ge2$.  Under the hypotheses of
\cref{lem:jl-f0-control},
\begin{align}
 \int w|\partial f|^4\phi
 \le C\int w\bigl[&e^{4f}\phi
       +\phi^{-3}|\partial\phi|^4
       +\varepsilon^{-2}\phi^{1-2\sigma}\notag\\
 &\quad+\varepsilon(|D|^2+|E|^2+|G|^2)
                       \phi^{1+\sigma}\bigr].
 \label{eq:jl-gradient-control}
\end{align}
\end{lemma}

\begin{proof}
The equation and the definitions of $D_\alpha,E_\alpha$ give the
pointwise identity
\begin{align}
 \Rea\bigl((w|\partial f|^2f_\alpha\phi)_{\bar\alpha}\bigr)
 ={}&w\bigl((n-1)|\partial f|^4
             -(n+1)e^{2f}|\partial f|^2\bigr)\phi\notag\\
 &+\Rea\bigl(wf_\alpha(\overline{D_\alpha}
                  +\overline{E_\alpha})\phi
       +w|\partial f|^2f_\alpha\phi_{\bar\alpha}\bigr).
 \label{eq:jl-gradient-divergence}
\end{align}
Integrate \eqref{eq:jl-gradient-divergence}.  For an auxiliary
$\rho>0$, Young's inequality bounds the first and last terms on the
right by
\begin{align*}
 e^{2f}|\partial f|^2\phi
 &\le C_\rho e^{4f}\phi+\rho|\partial f|^4\phi,\\
 |\partial f|^3|\partial\phi|
 &\le 3\rho|\partial f|^4\phi
       +C_\rho\phi^{-3}|\partial\phi|^4.
\end{align*}
The remaining term is bounded by the same weighted three-factor Young
inequality used in \cref{lem:jl-f0-control}.  Choose $\rho$ small
relative to $n-1$ and absorb the gradient terms.  This proves
\eqref{eq:jl-gradient-control}.
\end{proof}

\subsection{The annular estimate}

\begin{proposition}[The $m=0$ annular cutoff inequality]
\label{prop:m0-annular-cutoff}
Let $\eta=\eta_{a,R}$ be a smooth cutoff such that
\begin{equation}\label{eq:jl-cutoff-properties}
 0\le\eta\le1,\qquad
 \eta=1\ \hbox{on }B_R(a),\qquad
 \supp\eta\subset B_{2R}(a),\qquad
 |\partial\eta|\le CR^{-1}.
\end{equation}
Assume $n\ge2$.  For $R>0$,
\begin{align}
 \left(\int\fA_f\eta^8\right)^2
 \le{}&CR^{-2}
 \left(\int_{B_{2R}(a)}w(e^{4f}+R^{-4})\right)
 \left(\int_{B_{2R}(a)\setminus B_R(a)}
                    \fA_f\eta^8\right).
 \label{eq:m0-annular-cutoff}
\end{align}
The constant is independent of $a$ and $R$.
\end{proposition}

\begin{proof}
Apply \cref{lem:jl-basic-cutoff} with $\phi=\eta^8$.  Writing
\[
 I=\int\fA_f\eta^8,\qquad
 I_{\rm ann}=\int_{B_{2R}(a)\setminus B_R(a)}\fA_f\eta^8,
\]
we obtain
\begin{equation}\label{eq:pre-annular-cutoff}
 I^2\le CR^{-2}
 \left(\int w(e^{4f}+|\partial f|^4+f_0^2)\eta^6\right)
 I_{\rm ann}.
\end{equation}
Use \cref{lem:jl-f0-control,lem:jl-gradient-control} with
\begin{equation}\label{eq:jl-parameter-choice}
 \phi=\eta^6,\qquad \sigma=\frac13,\qquad
 \varepsilon=\rho R^2.
\end{equation}
The three relevant weights are exactly
\[
 \phi^{-3}|\partial\phi|^4
       \le CR^{-4}\eta^2,\qquad
 \phi^{1-2\sigma}=\eta^2,\qquad
 \phi^{1+\sigma}=\eta^8.
\]
Consequently,
\begin{align}
 \int w(|\partial f|^4+f_0^2)\eta^6
 \le C_\rho\int_{B_{2R}(a)}w(e^{4f}+R^{-4})
       +C\rho R^2 I.
 \label{eq:lower-order-cutoff-bound}
\end{align}
Insert \eqref{eq:lower-order-cutoff-bound} into
\eqref{eq:pre-annular-cutoff}.  Since $I_{\rm ann}\le I$, the last
term is at most $C\rho I^2$ and is absorbed by fixing $\rho>0$
sufficiently small.  This proves \eqref{eq:m0-annular-cutoff}.

The exponent $8$ is not essential.  The same calculation works with
any fixed $\theta>6$: the two error weights are then
$\eta^{\theta-6}$, while the absorbable defect term has weight
$\eta^\theta$.  Left translation and Heisenberg dilation give the
uniformity in $a$ and $R$.
\end{proof}

\bibliographystyle{amsplain}
\bibliography{references}

\end{document}